\documentclass[final,leqno]{siamltex}
\usepackage{amsmath,amssymb,epsfig}

\usepackage[latin1]{inputenc}
\usepackage{psfrag}
\usepackage{graphicx}
\usepackage{amsfonts}

\newcommand{\cgq}{$\mathrm{cG}(q)$}
\newcommand{\dgq}{$\mathrm{dG}(q)$}
\newcommand{\mcgq}{$\mathrm{mcG}(q)$}
\newcommand{\mdgq}{$\mathrm{mdG}(q)$}

\newcommand{\OD}[2]{\frac{d#1}{d#2}}

\newcommand{\PD}[2]{\frac{\partial#1}{\partial#2}}
\newcommand{\real}{\mathbb{R}}

\newtheorem{remark}{{\it Remark}}[section]

\title{Multi-Adaptive Galerkin Methods for ODEs I\thanks{Received by the editors
May 23, 2001; accepted for publication (in revised form) November 13, 2002;
published electronically May 2, 2003.
\URL sisc/24-6/38972.html}}

\author{Anders Logg\thanks{Department of Computational Mathematics,
        Chalmers University of Technology,
        SE--412 96 Göteborg, Sweden
        (logg@math.chalmers.se).}}

\begin{document}

\maketitle
\vspace{-1in}
\slugger{sisc}{2003}{24}{6}{1879--1902}
\vspace{.7in}

\setcounter{page}{1879}


\begin{abstract}
  	We present \emph{multi-adaptive} versions of the standard
  	continuous and discontinuous Galerkin methods for ODEs.
	Taking adaptivity one step further, we allow for individual time-steps,
	order and quadrature, so that in particular each individual component
	has its own time-step sequence.
	This paper contains a description of the methods, an analysis
	of their basic	properties, and a posteriori error analysis.
	In the accompanying paper [A.~Logg, {\it SIAM J.~Sci.~Comput.,} submitted],
	we present
	adaptive algorithms for time-stepping and global error control
	based on the results of the current paper.
\end{abstract}

\begin{keywords}
	multi-adaptivity, individual time-steps, local time-steps, ODE,
	continuous Galerkin, discontinuous Galerkin,
	global error control, adaptivity, mcG(q), mdG(q)
\end{keywords}

\begin{AMS}
	65L05, 65L07, 65L20, 65L50, 65L60, 65L70
\end{AMS}

\begin{PII}
S1064827501389722
\end{PII}

\pagestyle{myheadings}
\thispagestyle{plain}
\markboth{ANDERS LOGG}{MULTI-ADAPTIVE GALERKIN METHODS FOR ODEs I}

\section{Introduction}

In this paper, we present
multi-adaptive Galerkin methods for
initial value problems for systems of ODEs of the form
\begin{equation}
  \left\{
    \begin{array}{rcl}
      \dot{u}(t) &=& f(u(t),t),\qquad t\in(0,T], \\
      u(0) &=& u_0,
    \end{array}
  \right.
  \label{eq:u'=f}
\end{equation}
where $u : [0,T] \rightarrow \real^N$,
      $f : \real^N \times (0,T] \rightarrow \real^N$ is a given bounded
		function that is Lipschitz-continuous in $u$,
      $u_0 \in \real^N$ is a given initial condition, and
      $T>0$ is a given final time.
We use the term \emph{multi-adaptivity} to describe methods with individual time-stepping
for the different components $u_i(t)$ of the solution vector $u(t)=(u_i(t))$, including
({i}) time-step length,
({ii}) order, and
({iii})
quadrature, all chosen adaptively in a computational feedback process.
In the companion paper \cite{logg:multiadaptivity:II}, we apply the multi-adaptive methods to
a variety of problems to illustrate the potential of multi-adaptivity.

The ODE (\ref{eq:u'=f}) models a very large class of problems, covering many areas of applications.
Often different solution components have different time-scales and thus ask
for individual time-steps.
A prime example to be studied in detail below is our own solar system, where the moon
orbits around Earth once every month, whereas the period of Pluto is
250 years. In numerical simulations of the solar system, the time-steps
needed to track the orbit of the moon accurately are thus much less
than those required for Pluto, the difference
in time-scales being roughly a factor 3,000.

Surprisingly, individual time-stepping for ODEs has received little attention in the
large literature on numerical methods for ODEs; see, e.g.,
\cite{dahlquist:thesis,hairerwanner:book1,hairerwanner:book2,butcher:book,shampine:book}.
For specific applications, such as the $n$-body problem, methods with
individual time-stepping have been used---see, e.g.,
\cite{makino:local,alexander:local,dave:local} or \cite{kessel:local}---but
a general methodology has been
lacking. Our aim is to fill this gap.
For time-dependent PDEs, in particular for conservation laws of the type
$\dot{u} + f(u)_x = 0$, attempts have been made to construct methods with individual
(locally varying in space) time-steps. Flaherty et al. \cite{FlaLoy97} have constructed a method
based on the discontinuous Galerkin method combined with local forward Euler time-stepping.
A similar approach is taken in \cite{DawKir01}, where a method based on the original work by
Osher and Sanders \cite{OshSan83} is
presented for conservation laws in one and two space dimensions. Typically the time-steps used
are based on local CFL conditions rather than error estimates for the global error and
the methods are low order in time (meaning $\leq 2$).
We believe that our work on multi-adaptive Galerkin methods (including error estimation and arbitrary order methods)
presents a general methodology to
individual time-stepping, which will result in efficient integrators also for
time-dependent PDEs.

The methods presented in this paper fall within the general framework of adaptive
Galerkin methods based on piecewise polynomial approximation (finite element methods)
for differential equations, including the continuous Galerkin method \cgq\ of order
$2q$, and the discontinuous Galerkin method \dgq\ of order $2q+1$;
more precisely, we extend the \cgq\ and \dgq\ methods to their multi-adaptive
analogues \mcgq\ and \mdgq.
Earlier work on adaptive error control for the \cgq\ and \dgq\ methods include
\cite{delfour:dg,ejt:dg,claes:dg,estep:cg,estep:dg,EstWil96}.
The techniques for error analysis used in these references,
developed by Johnson and coworkers
(see, e.g., \cite{EJ:parab:I,EJ:parab:II,EJ:parab:III,EJ:parab:IV,EJ:parab:V,EJ:parab:VI},
and \cite{EJ:actanumerica} in particular) naturally carries over to the multi-adaptive
methods.

The outline of the paper is as follows: In section \ref{sec:features} we summarize the key
features of the multi-adaptive methods, and in section \ref{sec:comparison} we discuss
the benefits of the new methods in comparison to standard ODE codes.
We then motivate and present the formulation of the multi-adaptive methods
\mcgq\ and \mdgq\ in section \ref{sec:method}. Basic properties of these methods, such as
order, energy conservation, and monotonicity, are discussed in section \ref{sec:properties}.
In the major part of this paper, section \ref{sec:aposteriori}, we derive a posteriori error
estimates for the two methods based on duality arguments, including Galerkin errors, numerical errors,
and quadrature errors. We also prove an a posteriori error estimate for stability factors computed
from approximate dual solutions.

\section{Key features}
\label{sec:features}

We summarize the key features of our work on the \mcgq\ and \mdgq\ methods
as follows.

\subsection{Individual time-steps and order}
\label{sec:individual}
To discretize (\ref{eq:u'=f}), we introduce for each component, $i=1,\ldots,N$, a
partition of the time-interval $(0,T]$ into $M_i$ subintervals, $I_{ij} = (t_{i,j-1},t_{ij}]$,
$j=1,\ldots,M_i$,
and we seek an approximate solution $U(t)=(U_i(t))$ such that $U_i(t)$ is a polynomial of
degree $q_{ij}$ on every local interval $I_{ij}$.
Each
individual component $U_i(t)$ thus has its own sequence of time-steps,
$\{k_{ij}\}_{j=1}^{M_i}$.
The entire collection of individual time-intervals $\{I_{ij}\}$ may be organized
into a sequence of \emph{time-slabs}, collecting the time-intervals between
certain synchronised time-levels common to all components,
as illustrated in Figure \ref{fig:intervals}.

\begin{figure}[t]
	\begin{center}
		\leavevmode
		\psfrag{I}{$I_{ij}$}
		\psfrag{i}{$i$}
		\psfrag{k}{\hspace{-0.2cm}$k_{ij}$}
		\psfrag{0}{$0$}
		\psfrag{T}{$T$}
		\includegraphics[width=12cm]{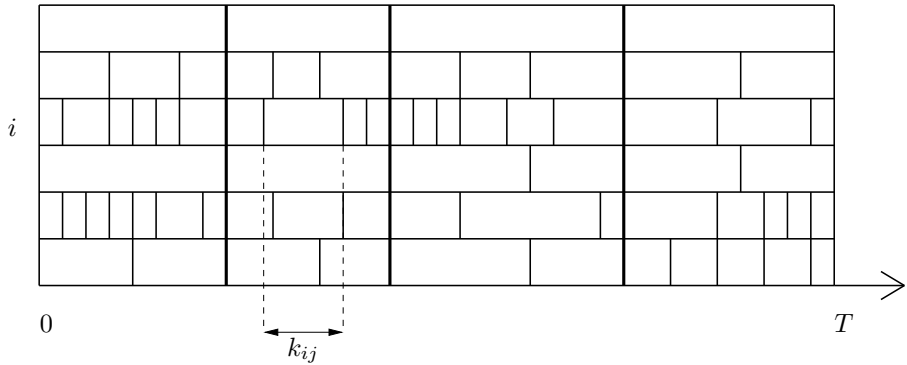}
		\caption{Individual time-discretizations for different components.}
		\label{fig:intervals}
	\end{center}
\end{figure}

\subsection{Global error control}
Our goal is to compute an approximation $U(T)$ of the exact solution $u(T)$ at final time
$T$ within a given tolerance $\mathrm{TOL}>0$, using a minimal amount of computational work.
This goal includes an aspect of \emph{reliability} (the error should be less than
the tolerance) and an aspect of \emph{efficiency} (minimal computational work).
To measure the error we choose a norm, such as the Euclidean norm $\|\cdot\|$ on
$\real^N$, or more generally some other quantity of interest (see \cite{sz:localerror}).

The mathematical basis of global error control in $\|\cdot\|$ for \mcgq\
is an error
representation of the form
\begin{equation}
	\| U(T) - u(T) \| = \int_0^T (R,\varphi) \ dt,
	\label{eq:inl,errorrepr}
\end{equation}
where $R = (R_i) = R(U,t) = \dot{U}(t) - f(U(t),t)$ is the \emph{residual} vector of the approximate
solution $U(t)$, $\varphi(t)$ is the solution of an associated linearized dual
problem, and $(\cdot,\cdot)$ is the $\real^N$ scalar product.

Using the Galerkin orthogonality, the error representation can be
converted into an error bound of the form
\begin{equation}
	\|U(T)-u(T)\| \leq \sum_{i=1}^N S_i(T) \max_{0\leq t \leq T} k_i(t)^{q_i(t)} |R_i(U,t)|,
	\label{eq:inl,errorestimate}
\end{equation}
where $\{S_i(T)\}_{i=1}^N$ are \emph{stability factors} for the different
components, depending on the dual solution $\varphi(t)$, and where
$k_i(t)=k_{ij}$, $q_i(t)=q_{ij}$ for $t\in I_{ij}$.
The error bound may take different forms depending on
how $\int_0^T (R,\varphi) \ dt$ is bounded in terms of $R$ and $\varphi$.

By solving the dual problem numerically, the individual stability factors $S_i(T)$
may be determined approximately, and thus the right-hand side of (\ref{eq:inl,errorestimate})
may be evaluated. The adaptive algorithm seeks to satisfy the \emph{stopping criterion}
\begin{equation}
	\sum_{i=1}^N S_i(T) \max_{0\leq t \leq T} k_i(t)^{q_i(t)} |R_i(U,t)| \leq \mathrm{TOL},
\end{equation}
with maximal time-steps $k = ( k_i(t) )$.

\subsection{Iterative methods}

Both \mcgq\ and \mdgq\ give rise to systems of nonlinear algebraic equations, coupling the values
of $U(t)$ over
each time-slab. Solving these systems with full Newton may be quite heavy, and we
have instead successfully used diagonal Newton methods of more explicit nature.

\subsection{Implementation of higher-order methods}
We have implemented\break \mcgq\ and \mdgq\ in
C++ for arbitrary $q$, which in practice means
$2q\leq 50$.
The implementation, \emph{Tanganyika}, is described in more detail
in \cite{logg:multiadaptivity:II} and
is publicly (GNU GPL) available for Linux/Unix \cite{logg:tanganyika}.

\subsection{Applications}

We have applied \mcgq\ and \mdgq\ to a variety of problems to illustrate
their potential; see \cite{logg:multiadaptivity:II}.
(See also \cite{logg:exjobb} and \cite{logg:oxford}.)
In these applications, including the Lorenz system, the solar system, and
a number of time-dependent PDE problems, we demonstrate the use of individual
time-steps, and for each system we solve the dual problem to collect extensive
information about the problems stability features,
which can be used for global error control.

\section{Comparison with standard ODE codes}
\label{sec:comparison}

Standard ODE codes use time-steps which are variable in time but the
same for all components, and the time-steps are adaptively
chosen by keeping the ``local error'' below a given local error tolerance
set by the user.
The global error connects to the local error through an estimate,
corresponding to (\ref{eq:inl,errorestimate}), of the form
\begin{equation}
	\{\mbox{global error}\} \leq S \ \max \{\mbox{local error}\},
	\label{eq:estimate,localerror}
\end{equation}
where $S$ is a stability factor.
Standard codes do not compute $S$,
which means that the connection between the
global error and the local error is left to be determined by the clever user,
typically by
computing with a couple of different tolerances.

Comparing the adaptive error control of standard ODE codes with the error control
presented in this paper and the accompanying paper \cite{logg:multiadaptivity:II},
an essential difference is thus the technique to estimate
the global error: either
by clever trial-and-error or, as we prefer,
by solving the dual problem and computing the stability factors.
Both approaches carry extra costs and what is best may be debated;
see, e.g., \cite{sz:localerror} for a comparison.

However, expanding the scope to multi-adaptivity with individual stability factors
for the different components, trial-and-error becomes
very difficult or impossible, and the methods for adaptive time-stepping and error control presented below
based on solving the dual problem seem to bring clear advantages in efficiency
and reliability.

For a presentation of the traditional approach to error estimation
in ODE codes, we refer to
\cite{petzold:book}, where the following rather pessimistic view is presented:
\emph{Here we just note that a precise error bound is often unknown and not really needed.}
We take the opposite view:
\emph{global error control is always needed and often possible to obtain at a reasonable cost}.
We hope that multi-adaptivity will bring new life to the discussion on efficient and reliable error control
for ODEs.

\section{Multi-adaptive Galerkin}
\label{sec:method}

In this section we present the multi-adaptive Galerkin methods,
\mcgq\ and \mdgq,
based on the discretization presented in section \ref{sec:individual}.

\subsection{The \protect\boldmath\mcgq\ method}
\label{sec:method,cg}

The \mcgq\ method for (\ref{eq:u'=f}) reads as follows:
Find $U\in V$ with $U(0)=u_0$, such that
\begin{equation}
  	 \int_0^T (\dot{U},v)\ dt = \int_0^T (f(U,\cdot),v) \ dt \qquad
    \forall v\in W,
    \label{eq:fem,mcg}
\end{equation}
where
\begin{equation}
  \begin{array}{rcl}
    V &=& \{v \in [C([0,T])]^N : v_i|_{I_{ij}}\in \mathcal{P}^{q_{ij}}(I_{ij}), \
    j=1,\ldots,M_i, \ i=1,\ldots,N \},\\
    W &=& \{v : v_i|_{I_{ij}}\in \mathcal{P}^{q_{ij}-1}(I_{ij}), \
    j=1,\ldots,M_i, \ i=1,\ldots,N \},\\
  \end{array}
  \label{eq:spaces,mcg}
\end{equation}
and where $\mathcal{P}^q(I)$ denotes the linear space of polynomials of degree $\leq q$ on $I$.
The trial functions in $V$ are thus continuous piecewise polynomials, locally
of degree $q_{ij}$,
and the test functions in $W$ are discontinuous piecewise polynomials that
are locally of degree $q_{ij}-1$.

Noting that the test functions are discontinuous, we can rewrite the global problem
(\ref{eq:fem,mcg}) as a number of successive local problems for each
component: For $i=1,\ldots,N$, $j=1,\ldots,M_i$,
find $U_i|_{I_{ij}}\in \mathcal{P}^{q_{ij}}(I_{ij})$ with $U_i(t_{i,j-1})$ given, such that
\begin{equation}
    \int_{I_{ij}} \dot{U}_i v \ dt = \int_{I_{ij}} f_i(U,\cdot) v \ dt \qquad
    \forall v\in \mathcal{P}^{q_{ij}-1}(I_{ij}).
    \label{eq:fem,mcg,local}
\end{equation}

We notice the presence of the vector $U(t)=(U_1(t),\ldots,U_N(t))$ in the local problem for $U_i(t)$ on $I_{ij}$. If
thus component $U_{i_1}(t)$ couples to component $U_{i_2}(t)$ through $f$, this means
that in order to solve the local problem for component $U_{i_1}(t)$ we need
to know the values of component $U_{i_2}(t)$ and vice versa.
The solution is thus implicitly defined by (\ref{eq:fem,mcg,local}).
Notice also that if we define the \emph{residual} $R$ of the approximate solution $U$
as $R_i(U,t) = \dot{U}_i(t) - f_i(U(t),t)$, we can rewrite (\ref{eq:fem,mcg,local}) as
\begin{equation}
    \int_{I_{ij}} R_i(U,\cdot) v \ dt = 0 \qquad
    \forall v\in \mathcal{P}^{q_{ij}-1}(I_{ij}),
    \label{eq:fem,mcg,local,orthogonality}
\end{equation}
i.e., the residual is orthogonal to the test space on every local interval. We refer
to this as the \emph{Galerkin orthogonality} for the \mcgq\ method.

Making an ansatz for every component $U_i(t)$ on every local interval $I_{ij}$ in terms of a
nodal basis for $\mathcal{P}^{q_{ij}}(I_{ij})$ (see the appendix), we can
rewrite (\ref{eq:fem,mcg,local}) as
\begin{equation}
   \xi_{ijm} =
   \xi_{ij0} +
   \int_{I_{ij}} w_m^{[q_{ij}]}(\tau_{ij}(t)) \ f_i(U(t),t) \ dt, \qquad m = 1,\ldots,q_{ij},
   \label{eq:mcg,xi,expl}
\end{equation}
where $\{\xi_{ijm}\}_{m=0}^{q_{ij}}$ are the nodal degrees of freedom for $U_i(t)$ on the interval $I_{ij}$,
$\{w_{m}^{[q]}\}_{m=1}^q\subset\mathcal{P}^{q-1}(0,1)$ are corresponding
polynomial weight functions, and $\tau_{ij}$ maps $I_{ij}$ to $(0,1]$:
$\tau_{ij}(t)=(t-t_{i,j-1})/(t_{ij}-t_{i,j-1})$.
Here we assume that the solution is expressed in terms of a nodal basis with the end-points
included, so that by the continuity requirement $\xi_{ij0}=\xi_{i,j-1,q_{i,j-1}}$.

Finally, evaluating the integral in (\ref{eq:mcg,xi,expl}) using nodal
quadrature, we obtain a fully discrete scheme in the form of an implicit
Runge--Kutta method:
For $i=1,\ldots,N$, $j=1,\ldots,M_i$, find $\{\xi_{ijm}\}_{m=0}^{q_{ij}}$, with
$\xi_{ij0}$ given by the continuity requirement, such that
\begin{equation}
   \xi_{ijm} =
   \xi_{ij0} +
   k_{ij} \sum_{n=0}^{q_{ij}} w_{mn}^{[q_{ij}]} \ f_i(U(\tau_{ij}^{-1}(s_{n}^{[q_{ij}]})),\tau_{ij}^{-1}(s_n^{[q_{ij}]})),
   \quad m = 1,\ldots,q_{ij},
   \label{eq:mcg,xi,expl,quad}
\end{equation}
for certain weights $\{w_{mn}^{[q]}\}$
and certain nodal points $\{s_n^{[q]}\}$ (see the appendix).

\subsection{The \protect\boldmath\mdgq\ method}

The \mdgq\ method in local form, corresponding to
(\ref{eq:fem,mcg,local}), reads as follows:
For $i=1,\ldots,N$, $j=1,\ldots,M_i$, find $U_i|_{I_{ij}}\in \mathcal{P}^{q_{ij}}(I_{ij})$, such that
\begin{equation}
    [U_i]_{i,j-1} v(t_{i,j-1}^+) + \int_{I_{ij}} \dot{U}_i v \ dt = \int_{I_{ij}} f_i(U,\cdot) v \ dt \qquad
    \forall v\in \mathcal{P}^{q_{ij}}(I_{ij}),
    \label{eq:fem,mdg,local}
\end{equation}
where $[\cdot]$ denotes the jump,
i.e., $[v]_{ij} = v(t_{ij}^+) - v(t_{ij}^-)$, and the initial condition is specified
for $i=1,\ldots,N$, by $U_i(0^-) = u_i(0)$. On a global level, the trial and
test spaces are given by
\begin{equation}
    V = W = \{v : v_i|_{I_{ij}}\in \mathcal{P}^{q_{ij}}(I_{ij}),
    j=1,\ldots,M_i, \ i=1,\ldots,N \}.
  \label{eq:spaces,mdg}
\end{equation}
In the same way as for the continuous method, we define the residual $R$ of the approximate
solution $U$ as $R_i(U,t)=\dot{U}_i(t) - f_i(U(t),t)$, defined on the inner of every local
interval $I_{ij}$, and we rewrite (\ref{eq:fem,mdg,local}) in the form
\begin{equation}
    [U_i]_{i,j-1} v(t_{i,j-1}^+) + \int_{I_{ij}} R_i(U,\cdot) v \ dt = 0 \qquad
    \forall v\in \mathcal{P}^{q_{ij}}(I_{ij}).
    \label{eq:fem,mdg,local,orthogonality}
\end{equation}
We refer to this as the Galerkin orthogonality for the \mdgq\ method.
Notice that this is similar to (\ref{eq:fem,mcg,local,orthogonality}) if we extend the
integral in (\ref{eq:fem,mcg,local,orthogonality}) to include the left end-point of the
interval $I_{ij}$. (The derivative of the discontinuous solution is a Dirac delta function at the
end-point.)

Making an ansatz for the solution in terms of some nodal basis, we get,
as for the continuous method, the following explicit version of (\ref{eq:fem,mdg,local})
on every local interval:
\begin{equation}
    \xi_{ijm} =
    \xi_{ij0}^- +
    \int_{I_{ij}} w_m^{[q_{ij}]}(\tau_{ij}(t)) \ f_i(U(t),t) \ dt, \qquad m = 0,\ldots,q_{ij},
    \label{eq:mdg,xi,expl}
\end{equation}
or, applying nodal quadrature,
\begin{equation}
   \xi_{ijm} =
   \xi_{ij0}^- +
   k_{ij} \sum_{n=0}^{q_{ij}} w_{mn}^{[q_{ij}]} \ f_i(U(\tau_{ij}^{-1}(s_{n}^{[q_{ij}]})),\tau_{ij}^{-1}(s_n^{[q_{ij}]})), \qquad m = 0,\ldots,q_{ij},
   \label{eq:mdg,xi,expl,quad}
\end{equation}
where the weight functions, the nodal points, and the weights are not the same as for the continuous method.

\subsection{The multi-adaptive \protect\boldmath\mcgq-\mdgq\ method}

The discussion above for the two methods extends naturally to using
different methods for different components. Some of the components
could therefore be solved for using the \mcgq\ method, while for
others we use the \mdgq\ method.
We can even change methods between different intervals.

Although the formulation thus includes adaptive orders and methods, as well
as adaptive time-steps, our focus will be mainly on adaptive time-steps.

\subsection{Choosing basis functions and quadrature}

What remains in order to implement the two methods specified by
(\ref{eq:mcg,xi,expl,quad}) and (\ref{eq:mdg,xi,expl,quad}) is to choose basis functions and
quadrature. For simplicity and efficiency reasons, it is desirable to let the
nodal points for the nodal basis coincide with the quadrature points.
It turns out that for both methods, the \mcgq\ and the \mdgq\ methods, this
is possible to achieve in a natural way. We thus choose the $q+1$ \emph{Lobatto quadrature points} for
the \mcgq\ method, i.e., the zeros of $ x P_q(x) -  P_{q-1} (x)$, where $P_q$ is the $q$th-order
Legendre polynomial on the interval; for the \mdgq\ method, we choose the \emph{Radau quadrature points},
i.e., the zeros of $P_{q}(x) + P_{q+1}(x)$ on the interval (with time reversed so as to include the \emph{right} end-point).
See \cite{logg:lic:I} for a detailed discussion on this subject.
The resulting discrete schemes are related to the implicit Runge--Kutta methods referred to as Lobatto and Radau methods;
see, e.g., \cite{butcher:book}.

\section{Basic properties of the multi-adaptive Galerkin methods}
\label{sec:properties}

In this section we examine some
basic properties of the multi-adaptive methods,
including order, energy conservation, and monotonicity.

\subsection{Order}

The standard \cgq\ and \dgq\ methods are of order $2q$ and $2q+1$, respectively.
The corresponding properties hold for the multi-adaptive methods, i.e.,
\mcgq\ is of order $2q$ and \mdgq\ is of order $2q+1$, assuming that the exact solution
$u$ is smooth. We examine this in more detail in subsequent papers.

\subsection{Energy conservation for \protect\boldmath\mcgq}\enlargethispage*{12pt}

The standard \cgq\ method is\break energy-conserving for Hamiltonian systems.
We now prove that also the \mcgq\ method has this property,
with the natural restriction that we should use the same time-steps for every pair of positions and
velocities.
We consider a Hamiltonian system,
\begin{equation}
  \ddot{x} = -\nabla_x P(x),
  \label{eq:hamilton}
\end{equation}
on $(0,T]$ with $x(t)\in\real^N$, together with initial conditions for $x$ and $\dot{x}$. Here $\ddot{x}$ is the
acceleration, which by Newton's second law is balanced by the force
$F(x)=-\nabla_x P(x)$ for some potential field $P$.
With $u=x$ and $v=\dot{x}$ we rewrite (\ref{eq:hamilton}) as
\begin{equation}
  \left[
    \begin{array}{c}
      \dot{u} \\
      \dot{v}
    \end{array}
  \right] =
  \left[
    \begin{array}{c}
      v \\
      F(u)
    \end{array}
  \right] =
  \left[
    \begin{array}{c}
      f_u(v) \\
      f_v(u)
    \end{array}
  \right] = f(u,v).
  \label{eq:hamilton,u'=f}
\end{equation}
The total energy $E(t)$ is the sum of the kinetic energy $K(t)$ and
the potential energy $P(x(t))$,
\begin{equation}
  E(t) = K(t) + P(x(t)),
  \label{eq:energy}
\end{equation}
with
\begin{equation}
  K(t) = \frac{1}{2}\|\dot{x}(t)\|^2 = \frac{1}{2} \|v(t)\|^2.
  \label{eq:kinetic}
\end{equation}
Multiplying (\ref{eq:hamilton}) with $\dot{x}$ it is easy to see that
energy is conserved for the continuous problem, i.e.,
$E(t) = E(0)$ for all $t\in [0,T]$. We now prove the
corresponding property for the discrete \mcgq\ solution of (\ref{eq:hamilton,u'=f}).

\begin{theorem}\label{thm5.1}
  The multi-adaptive continuous Galerkin method conserves energy in
  the following sense: Let $(U,V)$ be the \mcgq\ solution to
  {\rm(\ref{eq:hamilton,u'=f})} defined by {\rm(\ref{eq:fem,mcg,local}).}
  Assume that the same time-steps are used for
  every pair of positions and corresponding velocities.
  Then at every synchronized time-level $\bar{t}$, such as, e.g., $T$, we have
  \begin{equation}
    K(\bar{t}) + P(\bar{t}) = K(0) + P(0),
    \label{eq:energyconservation}
  \end{equation}
  with $K(t)=\frac{1}{2}\|V(t)\|^2$ and $P(t) = P(U(t))$.
\end{theorem}

\begin{proof}
  If every pair of positions and velocities have the same
  time-step sequence, then we may choose $\dot{V}$ as a test function in the equations for $U$, to get
  \begin{displaymath}
    \int_0^{\bar{t}} (\dot{U},\dot{V}) \ dt = \int_0^{\bar{t}} (V,\dot{V}) \ dt =
    \frac{1}{2} \int_0^{\bar{t}} \OD{}{t} \|V\|^2 \ dt = K(\bar{t}) - K(0).
  \end{displaymath}
  Similarly, $\dot{U}$ may be chosen as a test function in the equations for $V$ to get
  \begin{displaymath}
    \int_0^{\bar{t}} (\dot{V},\dot{U}) \ dt =
    \int_0^{\bar{t}} - \nabla P(U) \dot{U} \ dt =
    - \int_0^{\bar{t}} \OD{}{t} P(U) \ dt = - ( P(\bar{t}) - P(0) ),
  \end{displaymath}
  and thus $K(\bar{t}) + P(\bar{t}) = K(0) + P(0)$.
\qquad\end{proof}

\begin{remark}\label{remark5.1}\rm
  Energy conservation requires exact integration of the right-hand
  side $f$ (or at least that $\int_0^t (\dot{U},\dot{V}) \ dt + (P(t)  - P(0))= 0$)
  but can also be obtained in the case of quadrature;
  see \cite{Han01}.
\end{remark}

\subsection{Monotonicity}

We shall prove that the \mdgq\ method is $B$-stable (see \cite{butcher:book}).

\begin{theorem}\label{theorem5.2}
	Let $U$ and $V$ be the \mdgq\ solutions of {\rm(\ref{eq:u'=f})} with initial data
	$U(0^-)$ and $V(0^-)$, respectively,
	defined by {\rm(\ref{eq:fem,mdg,local})} on the same partition.
   If the right-hand side $f$ is monotone, i.e.,
	\begin{equation}
		(f(u,\cdot)-f(v,\cdot),u-v) \leq 0 \qquad \forall u,v \in \real^N,
	\end{equation}
	then, at every
	synchronized time-level $\bar{t}$, such as, e.g., $T$, we have
	\begin{equation}
		\|U(\bar{t}^-)-V(\bar{t}^-)\| \leq \|U(0^-)-V(0^-)\|.
	\end{equation}
\end{theorem}
\unskip
\begin{proof}
	Choosing the test function as $v = W = U - V$ in (\ref{eq:fem,mdg,local})
	for $U$ and $V$, summing over the local intervals, and subtracting the two equations,
	we have
	\begin{displaymath}
	\sum_{ij} \left[ [W_i]_{i,j-1} W_{i,j-1}^+ +
	                                     \int_{I_{ij}} \dot{W}_i W_i \ dt \right] =
	\int_0^T (f(U,\cdot)-f(V,\cdot),U-V) \ dt \leq 0.
	\end{displaymath}
	Noting that
	\begin{displaymath}
		\begin{array}{rcl}
			[W_i]_{i,j-1} W_{i,j-1}^+ + \int_{I_{ij}} \dot{W}_i W_i \ dt
			&=&
			\frac{1}{2}(W_{i,j-1}^+)^2 + \frac{1}{2}(W_{ij}^-)^2 - W_{i,j-1}^- W_{i,j-1}^+ \\
			&=&
			\frac{1}{2} [W_i]_{i,j-1}^2
			+ \frac{1}{2}\left( (W_{ij}^-)^2 - (W_{i,j-1}^-)^2 \right), \\
		\end{array}
	\end{displaymath}
	we get
	\begin{displaymath}
		-\frac{1}{2}\|W(0^-)\|^2 + \frac{1}{2}\|W(T^-)\|^2
		\leq
		\sum_{ij} [W_i]_{i,j-1} W_{i,j-1}^+ +
					 \int_{I_{ij}} \dot{W}_i W_i \ dt
		\leq 0,
	\end{displaymath}
	so that
	\begin{displaymath}
		\|W(T^-)\| \leq \|W(0^-)\|.
	\end{displaymath}
	The proof is completed noting that the same analysis applies with $T$ replaced by
	any other synchronized time-level $\bar{t}$.
\qquad\end{proof}

\begin{remark}\label{remark5.2}\rm
The proof extends to the fully discrete scheme, using the positivity of the
quadrature weights.
\end{remark}

\section{A posteriori error analysis}
\label{sec:aposteriori}

In this section we prove a posteriori error estimates for the
multi-adaptive Galerkin methods, including
quadrature and discrete solution errors.
Following the procedure outlined in the introduction, we first define
the dual linearized problem and then derive a representation formula
for the error in terms of the dual and the residual.

\subsection{The dual problem}
\label{sec:dual}

The dual problem comes in two different forms: a continuous and a
discrete.
For the {a posteriori} error analysis of this section, we
will make use of the continuous dual.
The discrete dual problem is used to prove a priori error estimates.

To set up the continuous dual problem, we
define, for given functions $v_1(t)$ and $v_2(t)$,
\begin{equation}
  J^*(v_1(t),v_2(t),t) =
  \left(
    \int_0^1 \PD{f}{u}(sv_1(t)+(1-s)v_2(t),t) \ ds
  \right)^*,
  \label{eq:J}
\end{equation}
where $^*$ denotes the transpose, and we note that
\begin{equation}
  \begin{array}{rcl}
    J(v_1,v_2,\cdot{})(v_1-v_2)
    &=& \int_0^1 \PD{f}{u}(sv_1+(1-s)v_2,\cdot{}) \ ds \ (v_1-v_2) \vspace{3pt}\\
    &=& \int_0^1 \PD{f}{s}(sv_1+(1-s)v_2,\cdot{}) \ ds
    = f(v_1,\cdot{}) - f(v_2,\cdot{}).
  \end{array}
	\label{eq:dualproperty}
\end{equation}
The continuous dual problem is then defined as the following system of ODEs:
\begin{equation}
  \left\{
    \begin{array}{rcl}
      -\dot{\varphi} &=& J^*(u,U,\cdot)\varphi + g \quad \mbox{ on } [0,T),\\
      \varphi(T) &=& \varphi_T,
    \end{array}
  \right.
  \label{eq:dual,continuous}
\end{equation}
with data $\varphi_T$ and right-hand side $g$.  Choosing the data and
right-hand side appropriately, we obtain error estimates for different
quantities of the computed solution.  We shall assume below that the
dual solution has $q$ continuous derivatives
($\varphi_i^{(q_{ij})}\in C(I_{ij})$ locally on interval $I_{ij}$)
for the continuous method and
$q+1$ continuous derivatives ($\varphi_i^{(q_{ij}+1)}\in C(I_{ij})$ locally on interval $I_{ij}$) for
the discontinuous method.

\subsection{Error representation}

The basis for the error analysis is the following error representation,
expressing the error of an approximate solution $U(t)$
in terms of the residual $R(U,t)$ via the dual solution $\varphi(t)$.
We stress that the result of the theorem is valid for any piecewise polynomial
approximation of the solution to the initial value problem $(\ref{eq:u'=f})$ and
thus in particular the \mcgq\ and \mdgq\ approximations.

\begin{theorem}\label{th:errorrepresentation,general}
	Let $U$ be a piecewise polynomial approximation of the exact solution $u$ of {\rm(\ref{eq:u'=f}),} and
	let $\varphi$ be the solution to
  {\rm(\ref{eq:dual,continuous})} with right-hand side $g(t)$ and initial
  data $\varphi_T$, and define the residual of the approximate solution
  $U$ as $R(U,t) = \dot{U}(t) - f(U(t),t)$,
  defined on the open intervals of the partitions $\cup_j I_{ij}$ as
  \begin{displaymath}
		R_i(U,t) = \dot{U}_i(t) - f_i(U(t),t), \qquad t \in (k_{i,j-1},k_{ij}),
  \end{displaymath}
  $j=1,\ldots,M_i$, $i=1,\ldots,N$. Assume also that $U$ is right-continuous at $T$.
  Then the error $e = U - u$ satisfies
  \begin{equation}
	 L_{\varphi_T,g}(e) \equiv
    (e(T),\varphi_T) + \int_0^T (e,g) \ dt =
    \sum_{i=1}^N \sum_{j=1}^{M_i}
    \left[
      \int_{I_{ij}} R_i(U,\cdot) \varphi_i \ dt +
      [U_i]_{i,j-1} \varphi_i(t_{i,j-1})
    \right].
    \label{eq:errorrepresentation,general}
  \end{equation}
\end{theorem}

\begin{proof}
  By the definition of the dual problem, we have using (\ref{eq:dualproperty})
  \begin{displaymath}
    \begin{array}{rcl}
      \int_0^T (e,g) \ dt
      &=& \int_0^T (e,- \dot{\varphi} - J^*(u,U,\cdot)\varphi) \ dt \\
      &=& \sum_{ij} \int_{I_{ij}} - e_i \dot{\varphi}_i \ dt +
          \int_0^T (-J(u,U,\cdot)e,\varphi) \ dt \\
      &=& \sum_{ij} \int_{I_{ij}} - e_i \dot{\varphi}_i \ dt +
          \int_0^T (f(u,\cdot)-f(U,\cdot),\varphi) \ dt \\
      &=& \sum_{ij} \int_{I_{ij}} - e_i \dot{\varphi}_i \ dt +
          \sum_{ij} \int_{I_{ij}} (f_i(u,\cdot)-f_i(U,\cdot)) \varphi_i \ dt.
    \end{array}
  \end{displaymath}
  Integrating by parts, we get
  \begin{displaymath}
    \int_{I_{ij}} - e_i \dot{\varphi}_i \ dt
    = e_i(t_{i,j-i}^+)\varphi(t_{i,j-1}) - e_i(t_{ij}^-)\varphi(t_{ij}) +
    \int_{I_{ij}} \dot{e}_i \varphi_i \ dt,
  \end{displaymath}
  so that
  \begin{displaymath}
    \begin{array}{rcl}
      \sum_{ij} \int_{I_{ij}} - e_i \dot{\varphi}_i \ dt
      &=& \sum_{ij} [e_i]_{i,j-1} \varphi_i(t_{i,j-1}) - (e(T^-),\varphi_T) +
          \int_0^T (\dot{e},\varphi) \ dt \\
      &=& \sum_{ij} [U_i]_{i,j-1} \varphi_i(t_{i,j-1}) - (e(T),\varphi_T) +
          \int_0^T (\dot{e},\varphi) \ dt. \\
        \end{array}
  \end{displaymath}
  Thus, with $L_{\varphi_T,g}(e) = (e(T),\varphi_T) + \int_0^T (e,g) \ dt$, we have
  \begin{displaymath}
    \begin{array}{rcl}
		L_{\varphi_T,g}(e)
      &=& \sum_{ij}
          \left[
             \int_{I_{ij}} (\dot{e}_i + f_i(u,\cdot)-f_i(U,\cdot)) \varphi_i \ dt +
             [U_i]_{i,j-1} \varphi_i(t_{i,j-1})
          \right] \vspace{3pt}\\
      &=& \sum_{ij}
          \left[
             \int_{I_{ij}} (\dot{U}_i - f_i(U,\cdot)) \varphi_i \ dt +
             [U_i]_{i,j-1} \varphi_i(t_{i,j-1})
          \right] \vspace{3pt}\\
      &=& \sum_{ij}
          \left[
             \int_{I_{ij}} R_i(U,\cdot) \varphi_i \ dt +
             [U_i]_{i,j-1} \varphi_i(t_{i,j-1})
          \right], \\
    \end{array}
  \end{displaymath}
  which completes the proof.
\qquad\end{proof}

We now apply this theorem to represent the error in various
norms. As before, we let $\|\cdot\|$ denote the Euclidean norm on
$\real^N$ and define $\|v\|_{L^1([0,T],\real^n)}=\int_0^T \|v\| \ dt$.

\begin{corollary}
	\label{cor:e(T)}
	If $\varphi_T=e(T)/{ \|e(T)\| }$ and $g = 0$, then
  \begin{equation}
    \|e(T)\| =
    \sum_{i=1}^N \sum_{j=1}^{M_i}
    \left[
      \int_{I_{ij}} R_i(U,\cdot) \varphi_i \ dt +
      [U_i]_{i,j-1} \varphi_i(t_{i,j-1})
    \right].
    \label{eq:errorrepresentation,finaltime}
  \end{equation}
\end{corollary}\unskip

\begin{corollary}
	\label{cor:L1}
	If $\varphi_T=0$ and $g(t) = e(t)/{ \|e(t)\| }$, then
  \begin{equation}
    \|e\|_{L^1([0,T],\real^N)} =
    \sum_{i=1}^N \sum_{j=1}^{M_i}
    \left[
      \int_{I_{ij}} R_i(U,\cdot) \varphi_i \ dt +
      [U_i]_{i,j-1} \varphi_i(t_{i,j-1})
    \right].
    \label{eq:errorrepresentation,L1}
  \end{equation}
\end{corollary}
\unskip
\subsection{Galerkin errors}

To obtain expressions for the Galerkin errors, i.e., the errors of the
\mcgq\ or \mdgq\ approximations, assuming exact quadrature and exact
solution of the discrete equations, we use two ingredients: the error
representation of Theorem \ref{th:errorrepresentation,general} and the
Galerkin orthogonalities, (\ref{eq:fem,mcg,local,orthogonality}) and
(\ref{eq:fem,mdg,local,orthogonality}).
We first prove the following interpolation estimate.

\begin{lemma}
	\label{lem:taylor}
  If $f\in C^{q+1}([a,b])$, then there is a constant $C_q$, depending only
  on $q$, such that
  \begin{equation}
    | f(x) - \pi^{[q]} f(x)|
    \leq
    C_q k^{q+1} \frac{1}{k} \int_a^b | f^{(q+1)}(y) | \ dy \qquad \forall x\in [a,b],
    \label{eq:interp,taylor}
  \end{equation}
  where $\pi^{[q]} f(x)$ is the $q$th-order Taylor expansion of $f$
  around $x_0=(a+b)/2$, $k=b-a$, and $C_q = 1/(2^q q!)$.
\end{lemma}

{\it Proof.}
  Using Taylor's formula with the remainder in integral form, we have
\begin{eqnarray*}
| f(x) - \pi^{[q]} f(x) | &=& \Bigg| \frac{1}{q!}\int_{x_0}^x f^{(q+1)}(y)(y-x_0)^{(q)} \ dy \Bigg| \\
&\leq& \frac{1}{2^qq!} k^{q+1} \frac{1}{k}\int_a^b |f^{(q+1)}(y)| \ dy. \qquad\endproof
\end{eqnarray*}

Note that since we allow the polynomial degree to change between different components
and between different intervals, the interpolation constant will change in the same way.
We thus have $C_{q_i}=C_{q_i}(t)=C_{q_{ij}}$ for $t\in I_{ij}$.

We can now prove a posteriori error estimates for the \mcgq\ and \mdgq\ methods.
The estimates come in a number of different versions. We typically use
$E_2$ or $E_3$ to adaptively determine the time-steps and $E_0$
or $E_1$ to
evaluate the error. The quantities $E_4$ and $E_5$ may be used for qualitative
estimates of error growth. We emphasize that all of the estimates derived in
Theorems \ref{th:aposteriori,cg} and \ref{th:aposteriori,dg} below may be of use
in an actual implementation, ranging from the very sharp estimate $E_0$ containing only local quantities
to the more robust estimate $E_5$ containing only global quantities.

\begin{theorem}
	\label{th:aposteriori,cg}
  The \mcgq\ method satisfies the following estimates:
  \begin{equation}
    |L_{\varphi_T,g}(e)| = E_0 \leq E_1 \leq E_2 \leq E_3 \leq E_4
    \label{eq:aposteriori,cg,1}
  \end{equation}
  and
  \begin{equation}
    |L_{\varphi_T,g}(e)| \leq E_2 \leq E_5,
    \label{eq:aposteriori,cg,2}
  \end{equation}
  where
  \begin{equation}
    \begin{array}{rcl}
      E_0 &=&
      \left|
        \sum_{i=1}^N \sum_{j=1}^{M_i}
        \int_{I_{ij}} R_i(U,\cdot) (\varphi_i - \pi_k \varphi_i) \ dt
      \right|,
      \\[1pt]
      E_1 &=&
      \sum_{i=1}^N \sum_{j=1}^{M_i}
      \int_{I_{ij}} |R_i(U,\cdot)| |\varphi_i - \pi_k \varphi_i| \ dt,
      \\[1pt]
      E_2 &=&
      \sum_{i=1}^N \sum_{j=1}^{M_i}
      C_{q_{ij}-1}  k_{ij}^{q_{ij}+1} \
      r_{ij} s_{ij}^{[q_{ij}]},
      \\[1pt]
      E_3 &=&
      \sum_{i=1}^N
      S^{[q_i]}_i
      \max_{[0,T]} \left\{ C_{q_{i}-1} k_{i}^{q_{i}} r_{i} \right\},
      \\[1pt]
      E_4 &=&
      S^{[q],1} \sqrt{N} \max_{i,[0,T]} \left\{ C_{q_{i}-1} k_{i}^{q_{i}} r_{i} \right\},
      \\[1pt]
      E_5 &=&
      S^{[q],2} \| C_{q-1} k^q R(U,\cdot) \|_{L^2(\real^N \times [0,T])},
    \end{array}
    \label{eq:aposteriori,cg,E}
  \end{equation}
  with $C_q$ as in Lemma {\rm\ref{lem:taylor},}
  $k_i(t) = k_{ij}$, $r_i(t) = r_{ij}$, and $s_i^{[q_{i}]}(t) = s_{ij}^{[q_{ij}]}$ for $t\in I_{ij}$,
  \begin{equation}
    \begin{array}{rclrcl}
      r_{ij}            &=& \frac{1}{k_{ij}} \int_{I_{ij}} |R_i(U,\cdot)| \ dt, &
      s_{ij}^{[q_{ij}]} &=& \frac{1}{k_{ij}} \int_{I_{ij}} |\varphi^{(q_{ij})}| \ dt, \vspace{3pt}\\
      S^{[q_i]}_i       &=& \int_0^T |\varphi_i^{(q_i)}| \ dt, &
      S^{[q],1}         &=& \int_0^T \|\varphi^{(q)}\| \ dt, \\
      S^{[q],2}         &=& \left( \int_0^T \|\varphi^{(q)}\|^2 \ dt \right)^{1/2}, \\
    \end{array}
    \label{eq:aposteriori,cg,definitions}
  \end{equation}
  and where $\pi_k \varphi$ is any test space approximation of the dual solution $\varphi$.
  Expressions such as $C_{q-1}k^q R$ are defined componentwise, i.e.,
  $(C_{q-1} k^q R(U,\cdot))_i = C_{q_{ij}-1} k_{ij}^{q_{ij}} R_i(U,\cdot)$ for
  $t\in I_{ij}$.
\end{theorem}

\begin{proof}
  Using the error representation of Theorem \ref{th:errorrepresentation,general}
  and the Galerkin orthogonality
  (\ref{eq:fem,mcg,local,orthogonality}),
  noting that the jump terms disappear since $U$ is continuous, we
  have
  \begin{displaymath}
    |L_{\varphi_T,g}(e)| =
    \left|
      \sum_{i=1}^N \sum_{j=1}^{M_i}
      \int_{I_{ij}} R_i(U,\cdot) (\varphi_i - \pi_k \varphi_i) \ dt
    \right| = E_0,
  \end{displaymath}
  where $\pi_k \varphi$ is any test space approximation of $\varphi$.
  By the triangle inequality, we have
  \begin{displaymath}
    E_0
    \leq
    \sum_{i=1}^N \sum_{j=1}^{M_i}
    \int_{I_{ij}} |R_i(U,\cdot) (\varphi_i - \pi_k \varphi_i)| \ dt = E_1.
  \end{displaymath}
  Choosing $\pi_k \varphi_i$ as in Lemma \ref{lem:taylor}
  on every interval  $I_{ij}$,
  we have
\begin{eqnarray*}
  E_1 &\leq& \sum_{ij}C_{q_{ij}-1} k_{ij}^{q_{ij}}\int_{I_{ij}} |R_i(U,\cdot)| \ dt \
		\frac{1}{k_{ij}} \int_{I_{ij}} |\varphi_i^{(q_{ij})}| \ dt \\
      &=& \sum_{ij}C_{q_{ij}-1}  k_{ij}^{q_{ij}+1} \ r_{ij} s_{ij}^{[q_{ij}]} = E_2.
\end{eqnarray*}
  Continuing, we have
  \begin{displaymath}
    \begin{array}{rcl}
      E_2
		&\leq&
      \sum_{i=1}^N
		\max_{[0,T]} \left\{ C_{q_{i}-1} k_{i}^{q_{i}} r_{i} \right\}
		\sum_{j=1}^{M_i}
      k_{ij} s_{ij}^{[q_{ij}]} \\
		&=&
      \sum_{i=1}^N
		\max_{[0,T]} \left\{ C_{q_{i}-1} k_{i}^{q_{i}} r_{i} \right\}
		\sum_{j=1}^{M_i}
      \int_{I_{ij}} |\varphi_i^{(q_{ij})}| \ dt\\
		&=&
      \sum_{i=1}^N
		\max_{[0,T]} \left\{ C_{q_{i}-1} k_{i}^{q_{i}} r_{i} \right\}
		\int_0^T |\varphi_i^{(q_{i})}| \ dt \\
      &=&
      \sum_{i=1}^N
      S^{[q_i]}_i
      \max_{[0,T]} \left\{ C_{q_{i}-1} k_{i}^{q_{i}} r_{i} \right\}
      = E_3,
	 \end{array}
  \end{displaymath}
	and, finally,
  \begin{displaymath}
    \begin{array}{rcl}
		E_3
      &\leq&
      \max_{i,[0,T]} \left\{ C_{q_{i}-1} k_{i}^{q_{i}} r_{i} \right\}
      \sum_{i=1}^N
      \int_0^T |\varphi_i^{(q_i)}| \ dt \vspace{3pt}\\
      &\leq&
      \max_{i,[0,T]} \left\{ C_{q_{i}-1} k_{i}^{q_{i}} r_{i} \right\}
      \sqrt{N}
      \int_0^T \|\varphi^{(q)}\| \ dt \vspace{3pt}\\
      &=&
      \max_{i,[0,T]} \left\{ C_{q_{i}-1} k_{i}^{q_{i}} r_{i} \right\}
      \sqrt{N}
      S^{[q],1}
      = E_4.
    \end{array}
  \end{displaymath}
  As an alternative we can use Cauchy's inequality in a different way. Continuing from $E_2$, we have
  \begin{displaymath}
    \begin{array}{rcl}
      E_2
      &=&
      \sum_{i=1}^N \sum_{j=1}^{M_i}
      C_{q_{ij}-1}  k_{ij}^{q_{ij}+1} \
      r_{ij} s_{ij}^{[q_{ij}]} \\[1pt]
      &=&
      \sum_{i=1}^N \sum_{j=1}^{M_i}
      C_{q_{ij}-1}  k_{ij}^{q_{ij}}  s_{ij}^{[q_{ij}]}
		\int_{I_{ij}} |R_i(U,\cdot)| \ dt \\[1pt]
      &=&
      \sum_{i=1}^N
      \int_0^T
      C_{q_{i}-1} k_{i}^{q_{i}} |R_{i}(U,\cdot)| s_{i}^{[q_i]} \ dt\\[1pt]
		&=&
      \int_0^T
      (C_{q-1} k^q |R(U,\cdot)|,s^{[q]}) \ dt \\[1pt]
      &\leq&
      \int_0^T \| C_{q-1} k^q R(U,\cdot) \| \| s^{[q]} \| \ dt\\[1pt]
      &\leq&
      \left(
        \int_0^T \| C_{q-1} k^q R(U,\cdot) \|^2 \ dt
      \right)^{1/2}
      \left(
        \int_0^T \| s^{[q]} \|^2 \ dt
      \right)^{1/2},
    \end{array}
  \end{displaymath}
  where $|R(U,\cdot)|$ denotes the vector-valued function with components $|R|_i = |R_i| = |R_i(U,\cdot)|$.
  Noting now that $s$ is the $L^2$-projection of $|\varphi^{(q)}|$
  onto the piecewise constants on the partition, we have
  \begin{displaymath}
    \left(
      \int_0^T \| s^{[q]} \|^2 \ dt
    \right)^{1/2}
    \leq
    \left(
      \int_0^T \| \varphi^{(q)}\|^2 \ dt
    \right)^{1/2},
  \end{displaymath}
  so that
  \begin{displaymath}
    |L_{\varphi_T,g}(e)|
    \leq
    \| C_{q-1} k^q R(U,\cdot) \|_{L^2(\real^N \times [0,T])}
    \| \varphi^{[q]} \|_{L^2(\real^N \times [0,T])} = E_5,
  \end{displaymath}
  completing the proof.
\qquad\end{proof}

The proof of the estimates for the \mdgq\ method is obtained similarly.
Since in the discontinuous method the test functions are on every interval
of one degree higher order than the test functions in the continuous method,
we can choose a better interpolant. Thus, in view of Lemma \ref{lem:taylor},
we obtain an extra factor $k_{ij}$ in the error estimates.

\begin{theorem}
	\label{th:aposteriori,dg}
   The \mdgq\ method satisfies the following estimates:
  \begin{equation}
    |L_{\varphi_T,g}(e)| = E_0 \leq E_1 \leq E_2 \leq E_3 \leq E_4
    \label{eq:aposteriori,dg,1}
  \end{equation}
  and
  \begin{equation}
    |L_{\varphi_T,g}(e)| \leq E_2 \leq E_5,
    \label{eq:aposteriori,dg,2}
  \end{equation}
  where
  \begin{equation}
    \begin{array}{rcl}
      E_0 &=&
      \left|
		  \sum_{ij}
        \int_{I_{ij}} R_i(U,\cdot) (\varphi_i - \pi_k \varphi_i) \ dt +
        [U_i]_{i,j-1} (\varphi_i(t_{i,j-1}) - \pi_k \varphi_i(t_{i,j-1}^+))
      \right|,
      \\
      E_1 &=&
		\sum_{ij}
      \int_{I_{ij}} |R_i(U,\cdot)| |\varphi_i - \pi_k \varphi_i| \ dt +
      |[U_i]_{i,j-1}||\varphi_i(t_{i,j-1}) - \pi_k \varphi_i(t_{i,j-1}^+)|,
      \\
      E_2 &=&
		\sum_{i=1}^N \sum_{j=1}^{M_i}
      C_{q_{ij}}  k_{ij}^{q_{ij}+2} \
      \bar{r}_{ij} s_{ij}^{[q_{ij}+1]},
      \\
      E_3 &=&
      \sum_{i=1}^N
      S^{[q_i+1]}_i
      \max_{[0,T]} \left\{ C_{q_{i}} k_{i}^{q_{i}+1} \bar{r}_{i} \right\},
      \\
      E_4 &=&
      S^{[q+1],1} \sqrt{N} \max_{i,[0,T]} \left\{ C_{q_{i}} k_{i}^{q_{i}+1} \bar{r}_{i} \right\},
      \\
      E_5 &=&
      S^{[q+1],2} \| C_{q} k^{q+1} \bar{R}(U,\cdot) \|_{L^2(\real^N \times [0,T])},
    \end{array}
    \label{eq:aposteriori,dg,E}
  \end{equation}
  with
  \begin{equation}
      \bar{r}_{ij}     = \frac{1}{k_{ij}} \int_{I_{ij}} |R_i(U,\cdot)| \ dt +
                           \frac{1}{k_{ij}} | [U_i]_{i,j-1} |,  \qquad
        \bar{R}_{i}(U,\cdot) = |R_i(U,\cdot)| + \frac{1}{k_{ij}} | [U_i]_{i,j-1} |,
    \label{eq:aposteriori,dg,definitions}
  \end{equation}
  and we otherwise use the notation of Theorem {\rm\ref{th:aposteriori,cg}.}
\end{theorem}

\begin{proof}
  As in the proof for the continuous method, we use the error representation of
Theorem \ref{th:errorrepresentation,general} and the
  Galerkin orthogonality (\ref{eq:fem,mdg,local,orthogonality}) to
  get
  \begin{displaymath}
    |L_{\varphi_T,g}(e)| =
    \left|
		\sum_{ij}
      \int_{I_{ij}} R_i(U,\cdot) (\varphi_i - \pi_k \varphi_i) \ dt +
      [U_i]_{i,j-1} (\varphi_i(t_{i,j-1}) - \pi_k \varphi_i(t_{i,j-1}^+))
    \right| = E_0.
  \end{displaymath}
  By Lemma \ref{lem:taylor} we obtain
  \begin{displaymath}
    \begin{array}{rcl}
      E_0
      &\leq&
		\sum_{ij}
      \int_{I_{ij}} |R_i(U,\cdot)| |\varphi_i - \pi_k \varphi_i| \ dt +
      |[U_i]_{i,j-1}||\varphi_i(t_{i,j-1}) - \pi_k \varphi_i(t_{i,j-1}^+)|
      =
      E_1 \\

      &\leq&
		\sum_{ij}
      C_{q_{ij}} k_{ij}^{q_{ij}+1}
      \left(
        \int_{I_{ij}} |R_i(U,\cdot)| \ dt +
        |[U_i]_{i,j-1}|
      \right)
      \frac{1}{k_{ij}} \int_{I_{ij}} |\varphi_i^{(q_{ij}+1)}| \ dt \\
      &\leq&
		\sum_{ij}
      C_{q_{ij}} k_{ij}^{q_{ij}+2}
      \bar{r}_{ij}
      s_{ij}^{[q_{ij}+1]}
      =
      E_2.
    \end{array}
  \end{displaymath}
  Continuing now in the same way as for the continuous method, we have
  $E_2\leq E_3\leq E_4$ and $E_2\leq E_5$.
\qquad\end{proof}

\begin{remark}\rm
	When evaluating the expressions $E_0$ or $E_1$, the interpolant
	$\pi_k \varphi$ does not have to be chosen as in Lemma \ref{lem:taylor}.
	This is only a convenient way to obtain the interpolation constant.
	In section \ref{sec:evaluating} below we discuss a more convenient choice of
	interpolant.
\end{remark}

\begin{remark}\rm
  If we replace
  $\frac{1}{k_{ij}} \int_{I_{ij}} |R_i| \ dt$ by
  $\max_{I_{ij}} |R_i|$, we may replace $C_q$ by a smaller constant $C_q'$.
  The value of the constant thus depends on the
  specific way the residual is measured.
\end{remark}

\subsection{Computational errors}

The error estimates of Theorems \ref{th:aposteriori,cg} and
\ref{th:aposteriori,dg}
are based on the Galerkin orthogonalities
(\ref{eq:fem,mcg,local,orthogonality}) and (\ref{eq:fem,mdg,local,orthogonality}).
If the corresponding discrete equations are not solved
exactly,
there will be an additional contribution to the total error.
Although Theorem \ref{th:errorrepresentation,general} is still
valid, the first steps in Theorems \ref{th:aposteriori,cg} and \ref{th:aposteriori,dg}
are not. Focusing on the continuous method, the first step in the proof of
Theorem \ref{th:aposteriori,cg} is the subtraction of a test space interpolant. This
is possible, since by the Galerkin orthogonality we have
\begin{displaymath}
	\sum_{i=1}^N \sum_{j=1}^{M_i} \int_{I_{ij}}
      R_i(U,\cdot) \pi_k \varphi_i \ dt = 0
\end{displaymath}
for all test space interpolants $\pi_k \varphi$. If the residual is no longer
orthogonal to the test space, we add and subtract this term to get to the point
where the implications of Theorem \ref{th:aposteriori,cg} are valid for one of the terms.
Assuming now that $\varphi$ varies slowly on each subinterval, we
estimate the remaining extra term as follows:
\begin{equation}
  \begin{array}{rcl}
	 E_C &=&
	 \left|
      \sum_{i=1}^N \sum_{j=1}^{M_i} \int_{I_{ij}}
      R_i(U,\cdot) \pi_k \varphi_i \ dt
    \right|
	 \leq
      \sum_{i=1}^N \sum_{j=1}^{M_i} \left| \int_{I_{ij}}
      R_i(U,\cdot) \pi_k \varphi_i \ dt \right| \\
    &\approx&
    \sum_{i=1}^N \sum_{j=1}^{M_i}
    k_{ij} |\bar{\varphi}_{ij}|
    \frac{1}{k_{ij}} \left| \int_{I_{ij}} R_i(U,\cdot) \ dt \right|
    =
    \sum_{i=1}^N \sum_{j=1}^{M_i}
    k_{ij} |\bar{\varphi}_{ij}|
    |\mathcal{R}^{\mathcal{C}}_{ij}| \\
    &\leq&
    \sum_{i=1}^N \bar{S}_i^{[0]} \max_{j} |\mathcal{R}^{\mathcal{C}}_{ij}|, \\
  \end{array}
  \label{eq:approx,cg}
\end{equation}
where $\bar{\varphi}$ is a piecewise constant approximation of $\varphi$ (using, say, the mean values on
the local intervals),
\begin{equation}
  \bar{S}_i^{[0]} = \sum_{j=1}^{M_i} k_{ij} |\bar{\varphi}_{ij}|
  \approx \int_0^T |\varphi_i| \ dt = S_i^{[0]}
  \label{eq:approx,stability}
\end{equation}
is a stability factor, and
we define the
\emph{discrete} or \emph{computational} residual as
\begin{equation}
  \mathcal{R}^{\mathcal{C}}_{ij} = \frac{1}{k_{ij}} \int_{I_{ij}} R_i(U,\cdot) \  dt
  = \frac{1}{k_{ij}}
    \left(
      (\xi_{ijq} - \xi_{ij0}) - \int_{I_{ij}} f_i(U,\cdot) \ dt
    \right).
  \label{eq:discreteresidual,cg}
\end{equation}
More precise estimates may be used if needed.

For the \mdgq\ method, the situation is similar with the computational residual now
defined as
\begin{equation}
  \mathcal{R}^{\mathcal{C}}_{ij}
  = \frac{1}{k_{ij}}
    \left(
      (\xi_{ijq} - \xi_{ij0}^-) - \int_{I_{ij}} f_i(U,\cdot) \ dt
    \right).
  \label{eq:discreteresidual,dg}
\end{equation}

Thus, to estimate the computational error,
we evaluate the computational residuals and multiply with the computed stability factors.

\subsection{Quadrature errors}

We now extend our analysis to take into account also quadrature errors.
We denote integrals evaluated by quadrature with $\tilde{\int}$. Starting from
the error representation as before, we have for the \mcgq\ method
\begin{equation}
	\begin{array}{rcl}
		L_{\varphi_T,g}(e)
		&=&
		 \int_0^T (R,\varphi) \ dt \vspace{3pt}\\
		&=&
		 \int_0^T (R,\varphi-\pi_k \varphi) \ dt +
		  \int_0^T (R,\pi_k \varphi) \ dt \vspace{3pt}\\
		&=&
		 \int_0^T (R,\varphi-\pi_k \varphi) \ dt +
		  \tilde{\int}_0^T (R,\pi_k \varphi) \ dt +
   		\left[
				\int_0^T (R,\pi_k \varphi) \ dt -
				\tilde{\int}_0^T (R,\pi_k \varphi) \ dt
			\right] \vspace{3pt}\\
		&=&
		 \int_0^T (R,\varphi-\pi_k \varphi) \ dt +
		  \tilde{\int}_0^T (R,\pi_k \varphi) \ dt +
   		\left(
				\tilde{\int}_0^T - \int_0^T
			\right) (f(U,\cdot),\pi_k \varphi) \ dt
	\end{array}
\end{equation}
if the quadrature is exact for $\dot{U} v$ when $v$ is a test function.
The first term of this expression was estimated in Theorem
\ref{th:aposteriori,cg} and the second term is the computational error
discussed previously (where $\tilde{\int}$ denotes that in a real implementation,
(\ref{eq:discreteresidual,cg}) is evaluated using quadrature).
The third term is the quadrature error,
which may be nonzero even if $f$ is linear, if the time-steps are different for
different components.
To estimate the quadrature error, notice that
\begin{equation}
	\begin{array}{rcl}
		\left(\tilde{\int}_0^T - \int_0^T\right) (f(U,\cdot),\pi_k \varphi) \ dt
		&=&
		\sum_{ij}
		\left( \tilde{\int}_{I_{ij}} - \int_{I_{ij}} \right)
		f_i(U,\cdot) \pi_k \varphi_i \ dt \\
		&\approx&
		\sum_{ij}
		k_{ij} \bar{\varphi}_{ij}
		\mathcal{R}^{\mathcal{Q}}_{ij}
		\leq
		\sum_{i=1}^N \bar{S}_i^{[0]} \max_j
		|\mathcal{R}^{\mathcal{Q}}_{ij}|,
	\end{array}
\end{equation}
where $\{\bar{S}_i^{[0]}\}_{i=1}^N$ are the same stability factors as in the estimate for
the computational error and
\begin{equation}
	\mathcal{R}^{\mathcal{Q}}_{ij} = \frac{1}{k_{ij}}
	\left( \tilde{\int_{I_{ij}}} f_i(U,\cdot) \ dt -
		    \int_{I_{ij}} f_i(U,\cdot) \ dt
	\right)
	\label{eq:quadratureresidual}
\end{equation}
is the \emph{quadrature residual}.
A similar estimate holds for the \mdgq\ method.

We now make a few comments on how to estimate the quadrature residual.
The Lobatto quadrature of the \mcgq\ method is exact for polynomials of degree
less than or equal to $2q-1$, and we have an order $2q$ estimate for
$\tilde{\int} - \int$ in terms of $f^{(2q)}$, and so we make the assumption
$\mathcal{R}^{\mathcal{Q}}_{ij}\propto k_{ij}^{2q_{ij}}$. If, instead of using the standard quadrature rule
over the interval with quadrature residual
$\mathcal{R}^{\mathcal{Q}_0}_{ij}$, we divide the interval into $2^m$ parts and use
the quadrature on every interval, summing up the result, we will get a different
quadrature residual, namely
\begin{equation}
	\mathcal{R}^{\mathcal{Q}_m} =
	\frac{1}{k} C 2^m (k/2^m)^{2q+1}
	= 2^{m(-2q)} C k^{2q} =
	2^{-2q} \mathcal{R}^{\mathcal{Q}_{m-1}},
	\label{quadrature,r,cg}
\end{equation}
where we have dropped the $_{ij}$ subindices.
Thus, since
$|\mathcal{R}^{\mathcal{Q}_m}| \leq
 |\mathcal{R}^{\mathcal{Q}_m}-\mathcal{R}^{\mathcal{Q}_{m+1}}| +
 |\mathcal{R}^{\mathcal{Q}_{m+1}}| =
 |\mathcal{R}^{\mathcal{Q}_m}-\mathcal{R}^{\mathcal{Q}_{m+1}}| +
 2^{-2q}|\mathcal{R}^{\mathcal{Q}_{m}}|$,
we have the estimate
\begin{equation}
	|\mathcal{R}^{\mathcal{Q}_m}| \leq
	\frac{1}{1-2^{-2q}}
	|\mathcal{R}^{\mathcal{Q}_m}-\mathcal{R}^{\mathcal{Q}_{m+1}}|.
	\label{eq:quadratureestimate,cg}
\end{equation}
Thus, by computing the integrals at two or more dyadic levels,
we may estimate quadrature residuals and thus the
quadrature error.

For the \mdgq\ method the only difference is that the basic quadrature
rule is one order better, i.e., instead of $2q$ we have $2q+1$, so that
\begin{equation}
	|\mathcal{R}^{\mathcal{Q}_m}| \leq
	\frac{1}{1-2^{-1-2q}}
	|\mathcal{R}^{\mathcal{Q}_m}-\mathcal{R}^{\mathcal{Q}_{m+1}}|.
	\label{eq:quadratureestimate,dg}
\end{equation}

\subsection{Evaluating \protect\boldmath$E_G$}
\label{sec:evaluating}

We now present an approach to estimating the quantity $\int_0^T (R(U,\cdot),\varphi-\pi_k \varphi) \ dt$
by direct evaluation, with $\varphi$ a computed dual solution and $\pi_k \varphi$ a suitably
chosen interpolant.
In this way we avoid introducing interpolation constants and computing derivatives
of the dual. Note, however, that although we do not explicitly compute any derivatives
of the dual, the regularity assumed in section \ref{sec:dual} for the dual solution is still implicitly required
for the computed quantities to make sense.
Starting now with
\begin{equation}
	E_G =
	\left|
		\sum_{i=1}^N \sum_{j=1}^{M_i} \int_{I_{ij}}
		R_i(U,\cdot) (\varphi_i - \pi_k \varphi_i) \ dt
	\right|
	\label{eq:EG,1}
\end{equation}
for the continuous method,
we realize that the best possible choice of interpolant, if we want to
prevent cancellation, is to choose $\pi_k \varphi$ such that
$R_i(U,\cdot) (\varphi_i - \pi_k \varphi_i) \geq 0$ (or $\leq 0$) on every local interval $I_{ij}$.
With such a choice of interpolant, we would have
\begin{equation}
	E_G =
	\left|
		\sum_{i=1}^N \sum_{j=1}^{M_i} \int_{I_{ij}}
		R_i(U,\cdot) (\varphi_i - \pi_k \varphi_i) \ dt
	\right| =
		\sum_{i=1}^N \sum_{j=1}^{M_i} \alpha_{ij} \int_{I_{ij}}
		|R_i(U,\cdot) (\varphi_i - \pi_k \varphi_i)| \ dt
	\label{eq:EG,2}
\end{equation}
with $\alpha_{ij}=\pm 1$.
The following lemmas give us an idea of how to choose the
interpolant.

\begin{lemma}
	If, for $i=1,\ldots,N$, $f_i=f_i(U(t),t)=f_i(U_i(t),t)$ and $f_i$ is linear or, alternatively,
	$f=f(U(t),t)$ is linear and all components have the same time-steps and order,
	then every component $R_i(U,\cdot)$ of
	the \mcgq\ residual is a Legendre polynomial of order $q_{ij}$ on
	$I_{ij}$, for $j=1,\ldots,M_i$.
\end{lemma}

\begin{proof}
	On every interval $I_{ij}$ the residual component
	$R_i(U,\cdot)$	is orthogonal to $\mathcal{P}^{q_{ij}-1}(I_{ij})$.
	Since the conditions assumed in the statement of the lemma guarantee that
	the residual is a polynomial of degree $q_{ij}$ on every interval $I_{ij}$,
	it is clear that \linebreak on every such interval it is the $q_{ij}$th-order Legendre
	polynomial (or a multiple \linebreak thereof).
\qquad\end{proof}

Even if the rather strict conditions of this lemma do not hold, we can say something
similar. The following lemma restates this property in terms of approximations of
the residual.

\begin{lemma}
	\label{lem:legendre,appr}
	Let $\tilde{R}$ be the local $L^2$-projection of the \mcgq\ residual $R$ onto
	the trial space,
	i.e., $\tilde{R}_i(U,\cdot)|_{I_{ij}}$ is the $L^2(I_{ij})$-projection onto $\mathcal{P}^{q_{ij}}(I_{ij})$
	of $R_i(U,\cdot)|_{I_{ij}}$,
	$j=1,\ldots,M_i$, $i=1,\ldots,N$.
	Then every $\tilde{R}_i(U,\cdot)|_{I_{ij}}$ is a Legendre polynomial of degree $q_{ij}$.
\end{lemma}

\begin{proof}
	Since $\tilde{R}_i(U,\cdot)$ is the $L²$-projection of $R_i(U,\cdot)$ onto
	$\mathcal{P}^{q_{ij}}(I_{ij})$ on $I_{ij}$, we have
	\begin{displaymath}
		\int_{I_{ij}} \tilde{R}_i(U,\cdot) v \ dt =
		\int_{I_{ij}} R_i(U,\cdot) v \ dt = 0
	\end{displaymath}
	for all $v\in\mathcal{P}^{q_{ij}-1}(I_{ij})$, so that $\tilde{R}_i(U,\cdot)$ is the
	$q_{ij}$th-order Legendre polynomial on \linebreak$I_{ij}$.
\qquad\end{proof}

To prove the corresponding results for the discontinuous method, we first note
some basic properties of Radau polynomials.
\begin{lemma}
	\label{lem:radau}
	Let $P_q$ be the $q$th-order Legendre polynomial on $[-1,1]$.
	Then the $q$th-order Radau polynomial, $Q_q(x)=(P_{q}(x)+P_{q+1}(x))/(x+1)$, has
	the following property:
	\begin{equation}
		I = \int_{-1}^1 Q_q(x) (x+1)^p \ dx = 0
		\label{eq:radauequation}
	\end{equation}
	for $p=1,\ldots, q$.
	Conversely, if $f$ is a polynomial of degree $q$ on $[-1,1]$ and
	has the property {\rm(\ref{eq:radauequation}),} i.e.,
	$\int_{-1}^1 f(x) (x+1)^p \ dx = 0$ for $p=1,\ldots,q$, then $f$
	is a Radau polynomial.
\end{lemma}

\begin{proof}
	We can write the $q$th-order Legendre polynomial on $[-1,1]$ as
	$P_q(x) = \frac{1}{q!2^q} D^q((x^2-1)^q)$. Thus,
	integrating by parts, we have
	\begin{displaymath}
		\begin{array}{rcl}
			I &=&
			\int_{-1}^1 \frac{P_q(x)+P_{q+1}(x)}{x+1} (x+1)^p \ dx \\
			&=&
			\frac{1}{q! 2^q}
			\int_{-1}^1 D^q ( (x^2-1)^q + x(x^2-1)^q) (x+1)^{p-1} \ dx \\
			&=&
			\frac{1}{q! 2^q}
			\int_{-1}^1 D^q ( (x+1)(x^2-1)^q ) (x+1)^{p-1} \ dx \\
			&=&
			\frac{1}{q! 2^q}
			(-1)^{p}
			\int_{-1}^1 D^{q-p} ( (x+1)(x^2-1)^q ) D^p (x+1)^{p-1} \ dx
			=
			0,
		\end{array}
	\end{displaymath}
	since $D^l((x+1)(x²-1)^q)$ is zero at $-1$ and $1$ for $l<q$.
	Assume now that $f$ is a polynomial of degree $q$ on $[-1,1]$ with the
	property (\ref{eq:radauequation}). Since $\{(x+1)^p\}_{p=1}^q$ are
	linearly independent on $[-1,1]$ and orthogonal to the Radau polynomial
	$Q_q$, $\{Q_q(x),(x+1),(x+1)^2,\ldots,(x+1)^q\}$ form a basis for
	$\mathcal{P}^q([-1,1])$. If then $f$ is orthogonal to the subspace spanned by
	$\{(x+1)^p\}_{p=1}^q$, we must have $f=c Q_q$ for some constant $c$, and
	the proof is complete.
\qquad\end{proof}

\begin{lemma}
	If, for $i=1,\ldots,N$, $f_i=f_i(U(t),t)=f_i(U_i(t),t)$ and $f_i$ is linear or, alternatively,
	$f=f(U(t),t)$ is linear and all components have the same time-steps and order,
	then every component $R_i(U,\cdot)$ of
	the \mdgq\ residual is a Radau polynomial of order $q_{ij}$ on
	$I_{ij}$ for $j=1,\ldots,M_i$.
\end{lemma}

\begin{proof}
	Note first that by assumption the residual $R_i(U,\cdot)$ is a polynomial of degree
	$q_{ij}$ on $I_{ij}$.
	By the Galerkin orthogonality, we have
	\begin{displaymath}
		0 = \int_{I_{ij}} R_i(U,\cdot) v \ dt + [U_i]_{i,j-1} v(t_{i,j-1}^+)
		\qquad \forall v\in\mathcal{P}^{q_{ij}}(I_{ij}),
	\end{displaymath}
	which holds especially for $v(t) = (t-t_{i,j-1})^p$ with
	$p=1,\ldots,q$, for which the jump terms disappear.
	Rescaling to $[-1,1]$, it follows
	from Lemma \ref{lem:radau} that the residual $R_i(U,\cdot)$ must be a
	Radau polynomial on $I_{ij}$.
\qquad\end{proof}

Also for the discontinuous method there is a reformulation
in terms of approximations of the residual.

\begin{lemma}
	\label{lem:radau,appr}
	Let $\tilde{R}$ be the local $L^2$-projection of the \mdgq\ residual $R$ onto
	the trial space,
	i.e., $\tilde{R}_i(U,\cdot)|_{I_{ij}}$ is the $L^2(I_{ij})$-projection onto $\mathcal{P}^{q_{ij}}(I_{ij})$
	of $R_i(U,\cdot)|_{I_{ij}}$,
	$j=1,\ldots,M_i$, $i=1,\ldots,N$.
	Then every $\tilde{R}_i(U,\cdot)|_{I_{ij}}$ is a Radau polynomial of degree $q_{ij}$.
\end{lemma}

\begin{proof}
	Since $\tilde{R}_i(U,\cdot)$ is the $L²$-projection of $R_i(U,\cdot)$ onto
	$\mathcal{P}^{q_{ij}}(I_{ij})$ on $I_{ij}$, it follows
	from the Galerkin orthogonality that
	\begin{displaymath}
		\int_{I_{ij}} \tilde{R}_i(U,\cdot) v \ dt =
		\int_{I_{ij}} R_i(U,\cdot) v \ dt = 0
	\end{displaymath}
	for any $v(t) = (t-t_{i,j-1})^p$ with $1\leq p\leq q$.
	From Lemma \ref{lem:radau} it then follows that $\tilde{R}_i(U,\cdot)$
	is a Radau polynomial on $I_{ij}$.
\qquad\end{proof}

We thus know that the \mcgq\ residuals are (in the sense of Lemma \ref{lem:legendre,appr})
Legendre polynomials on the local intervals and that the \mdgq\ residuals
are (in the sense of Lemma \ref{lem:radau,appr}) Radau polynomials. This is illustrated in Figure
\ref{fig:residualproperties}.

\begin{figure}[t]
	\begin{center}
		\leavevmode
		\psfrag{5.5}{}
		\psfrag{t}{$t$}
		\psfrag{r}{$R_i(U,\cdot)$}
    	\includegraphics[width=7cm]{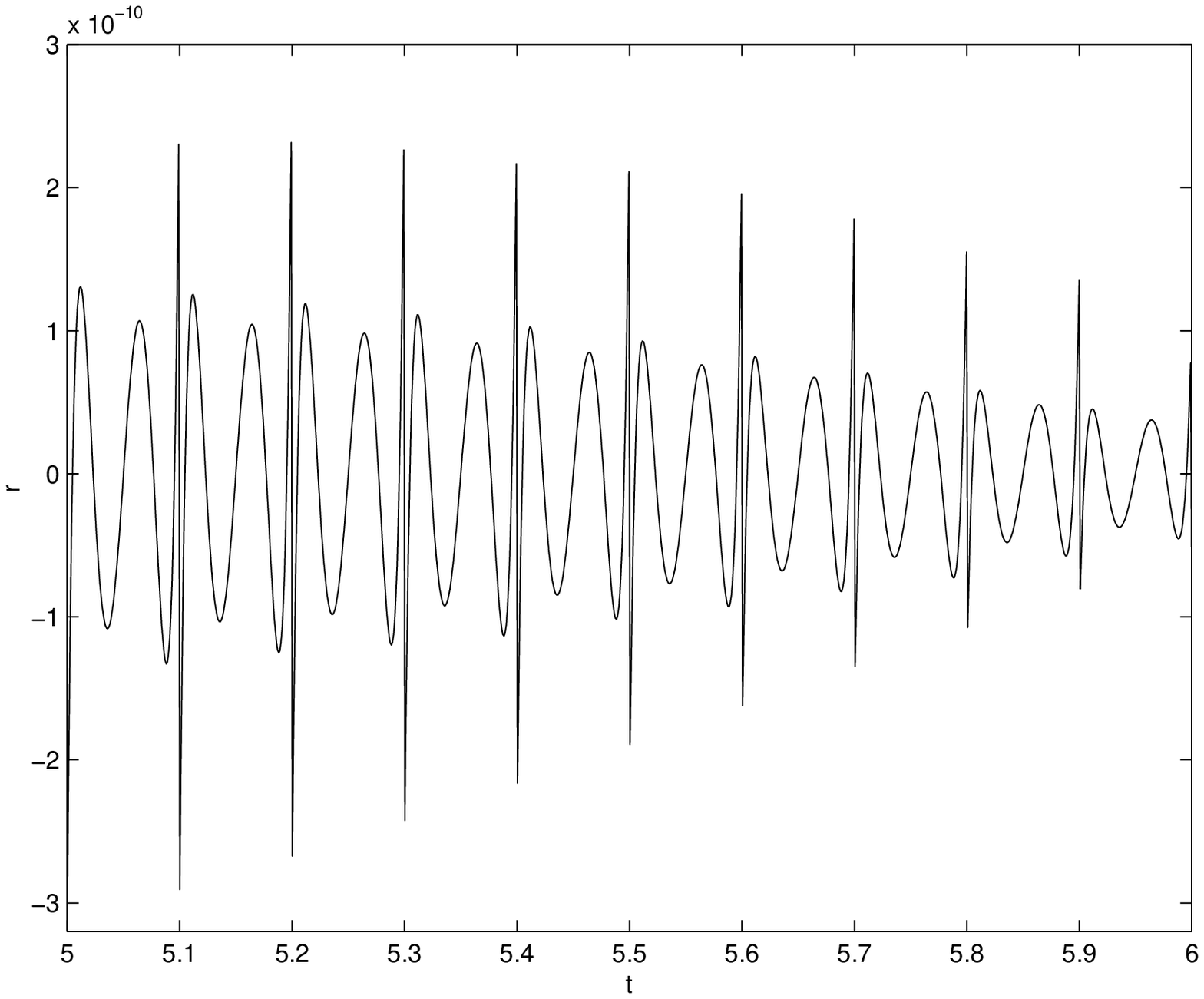}
    	\includegraphics[width=7cm]{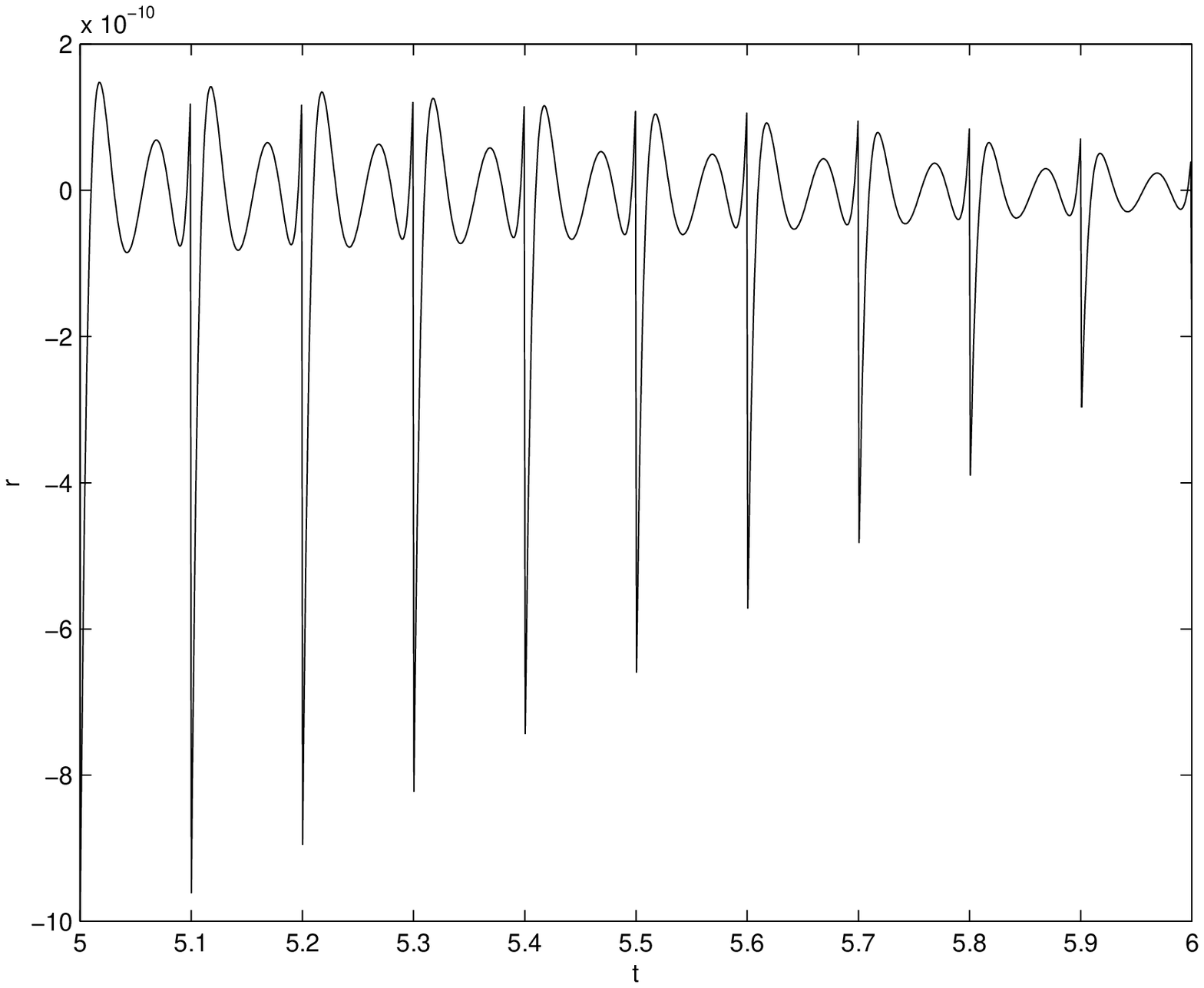}
    	\caption{The Legendre-polynomial residual of the \mcgq\ method (left)
					and the Radau-polynomial residual of the \mdgq\ method (right),
					for polynomials of degree five, i.e., methods of order $10$ and $11$,
					respectively.}
    	\label{fig:residualproperties}
  	\end{center}
\end{figure}

From this information about the residual, we now choose the interpolant.
Assume that the polynomial order of the method on some interval is $q$ for the
continuous method. Then the dual should be interpolated by a polynomial
of degree $q-1$, i.e., we have freedom to interpolate at exactly $q$ points.
Since a $q$th-order Legendre polynomial
has $q$ zeros on the interval, we may choose to interpolate the dual exactly at
those points where the residual is zero. This means that if the dual
can be approximated well enough by a polynomial of degree
$q$, the product $R_i(U,\cdot) (\varphi_i - \pi_k \varphi_i)$ does not change sign on
the interval.

For the discontinuous method, we should interpolate the dual with a polynomial
of degree $q$, i.e., we have freedom to interpolate at exactly $q+1$ points.
To get rid of the jump terms that are present in the error representation for
the discontinuous method, we want to interpolate the dual at the beginning of
every interval. This leaves $q$ degrees of freedom.
We then choose to interpolate the dual at the $q$ points within the interval
where the Radau polynomial is zero.

As a result, we may choose the interpolant in such a way that we have
\begin{equation}
	|L_{\varphi_T,g}(e)| =
	\left|
	\sum_{ij} \int_{I_{ij}} R_i(U,\cdot) (\varphi_i - \pi_k \varphi_i) \ dt
	\right| =
	\sum_{ij} \alpha_{ij}
	\int_{I_{ij}} |R_i(U,\cdot) (\varphi_i - \pi_k \varphi_i)| \ dt,
\end{equation}
with $\alpha_{ij}=\pm 1$,
for both the \mcgq\ method and the \mdgq\ method (but the interpolants are different). Notice that the jump terms
for the discontinuous method have disappeared.

There is now a simple way to compute the integrals
$\int_{I_{ij}} R_i(U,\cdot) (\varphi_i - \pi_k \varphi_i) \ dt$. Since the integrands are,
in principle, products of two polynomials for which we know the positions of the
zeros, the product is a polynomial with known properties.
There are then constants $C_q$ (which can be computed numerically), depending on the order and the method, such
that
\begin{equation}
	\int_{I_{ij}} |R_i(U,\cdot) (\varphi_i - \pi_k \varphi_i)| \ dt =
	C_{q_{ij}} k_{ij} |R_i(U,t_{ij}^-)||\varphi_i(t_{ij})-\pi_k \varphi_i(t_{ij}^-)|.
	\label{eg:errorestimate,computational,1}
\end{equation}

Finally, note that there are ``computational'' counterparts also for the estimates
of type $E_3$ in Theorems \ref{th:aposteriori,cg} and \ref{th:aposteriori,dg},
namely
\begin{equation}
	\begin{array}{rcl}
		|L_{\varphi_T,g}(e)|
		&\leq&
		\sum_{ij}
		\int_{I_{ij}} |R_i(U,\cdot)| |\varphi_i - \pi_k \varphi_i| \ dt \\
		&=&
		\sum_{ij}
			C_{q_{ij}}' k_{ij}^{q_{ij}} |R_i(U,t_{ij}^-)|
		\int_{I_{ij}} \frac{1}{k_{ij}^{q_{ij}}}|\varphi_i - \pi_k \varphi_i| \ dt \\
		&\leq&
		\sum_{i=1}^N
		\tilde{S}_i
		\max_{j=1,\ldots,M_i}
		C_{q_{ij}}' k_{ij}^{q_{ij}} |R_i(U,t_{ij}^-)|,
	\end{array}
	\label{eg:errorestimate,computational,2}
\end{equation}
with $\tilde{S}_i = \int_0^T \frac{1}{k_i^{q_{i}}}|\varphi_i - \pi_k \varphi_i| \ dt$
for the continuous method and similarly for the discontinuous
method.

\subsection{The total error}

The total error is composed of three parts---the Galerkin error,
$E_G$, the computational error, $E_C$
and the quadrature error, $E_Q$:
\begin{equation}
	|L_{\varphi_T,g}(e)| \leq E_G + E_C + E_Q.
	\label{eq:totalerror,general}
\end{equation}
As an example, choosing estimate $E_3$ of Theorems \ref{th:aposteriori,cg}
and \ref{th:aposteriori,dg} we have the following (approximate)
error estimate
for the \mcgq\ method:
\begin{equation}
	|L_{\varphi_T,g}(e)| \leq
   \sum_{i=1}^N
	\left[
   	S^{[q_i]}_i
   	\max_{[0,T]} \left\{ C_{q_{i}-1} k_{i}^{q_{i}} r_{i} \right\} +
    	\bar{S}^{[0]}_i \max_{[0,T]} |\mathcal{R}^{\mathcal{C}}_i| +
    	\bar{S}^{[0]}_i \max_{[0,T]} |\mathcal{R}^{\mathcal{Q}}_i|
	\right];
	\label{eq:totalerror,explicit,cg}
\end{equation}
for the \mdgq\ method we have
\begin{equation}
	|L_{\varphi_T,g}(e)| \leq
   \sum_{i=1}^N
	\left[
   	S^{[q_i+1]}_i
   	\max_{[0,T]} \left\{ C_{q_{i}} k_{i}^{q_{i}+1} \bar{r}_{i} \right\} +
    	\bar{S}^{[0]}_i \max_{[0,T]} |\mathcal{R}^{\mathcal{C}}_i| +
    	\bar{S}^{[0]}_i \max_{[0,T]} |\mathcal{R}^{\mathcal{Q}}_i|
	\right].
	\label{eq:totalerror,explicit,dg}
\end{equation}
These estimates containing Galerkin errors, computational errors, and
quadrature errors also include numerical round-off errors (included in the computational error).
\emph{Modelling errors} could also be similarly accounted for since these are closely related to quadrature errors,
in that both errors can be seen as arising from integrating the wrong right-hand side.

The true global error may thus be estimated in terms of computable stability factors
and residuals.
We expect the estimate for the Galerkin error, $E_G$, to be
quite sharp, while $E_C$ and $E_Q$ may be less sharp.
Even sharper estimates are obtained using estimates $E_0$, $E_1$, or $E_2$ of
Theorems \ref{th:aposteriori,cg} and \ref{th:aposteriori,dg}.

\subsection{An a posteriori error estimate for the dual}

We conclude this section by proving a computable a posteriori error estimate
for the dual problem.
To compute the stability factors used in the error estimates presented above,
we solve the dual problem numerically, and we
thus face the problem of estimating the error in the stability factors.

To demonstrate how relative errors of stability factors can be estimated using
the same technique as above, we compute the relative error for the stability
factor $S_{\varphi}(T)$, defined as
\begin{equation}
	S_{\varphi}(T) = \sup_{\|\varphi(T)\|=1} \int_0^T \|\varphi\| \ dt
	\label{eq:stability,max}
\end{equation}
for a computed approximation $\Phi$ of the dual solution $\varphi$.

To estimate the relative error of the stability factor, we use the error representation
of Theorem \ref{th:errorrepresentation,general} to represent the $L^1([0,T],\real^N)$-error
of $\Phi$ in terms of the residual of $\Phi$ and the dual of the dual, $\omega$.
In \cite{logg:lic:I} we prove the following lemma, from which the estimate follows.

\begin{lemma}
	\label{lem:dualdual}
	Let $\varphi$ be the dual solution with stability factor $S_{\varphi}(t)$,
	i.e., with data $\|\varphi(t)\|=1$ specified at time $t$,
	and
	let $\omega$ be the dual of the dual.
	We then have the following estimate:
	\begin{equation}
		\|\omega(t)\| \leq S_{\varphi}(T-t) \qquad \forall t\in [0,T].
	\end{equation}
\end{lemma}
\unskip

\begin{theorem}
	Let $\Phi$ be a continuous approximation of the dual solution
	with residual $R_{\Phi}$, and assume that
	$S_{\varphi}(t)/S_{\varphi}(T)$ is bounded by $C$ on $[0,T]$.
	Then the following estimate holds for the relative error of the stability factor
	$S_{\Phi}(T)$:
 	\begin{equation}
		|S_{\Phi}(T)-S_{\varphi}(T)|/S_{\varphi}(T) \leq
		C \int_0^T \| R_{\Phi} \| \ dt,
	\end{equation}
	and for many problems we may take $C=1$.
\end{theorem}

\begin{proof}
	By Corollary \ref{cor:L1},
	we have an expression for the $L^1([0,T],\real^N)$-error of the dual,
	so that
	\begin{equation}
		\begin{array}{rcl}
			|S_{\Phi}(T)-S_{\varphi}(T)|
			&=&
			\left|
				\int_0^T \|\Phi\| \ dt - \int_0^T \|\varphi\| \ dt
			\right| \\
			&=&
			\left|
				\int_0^T ( \|\Phi\| - \|\varphi\| ) \ dt
			\right|
			\leq
			\int_0^T \|\Phi - \varphi\| \ dt \\
			&=&
			\|\Phi - \varphi\|_{L^1([0,T],\real^n)}
			=
			\int_0^T (R_{\Phi},\omega(T-\cdot)) \ dt \\
			&\leq&
			\int_0^T \|R_{\Phi}\| \|\omega(T-\cdot)\| \ dt. \\
		\end{array}
	\end{equation}
	With $C$ defined as above it now follows by Lemma \ref{lem:dualdual}
	that
	\begin{displaymath}
		|S_{\Phi}(T)-S_{\varphi}(T)|	\leq
		C \int_0^T \|R_{\Phi}\| \ dt \ S_{\varphi}(T),
	\end{displaymath}
	and the proof is complete.
\qquad\end{proof}

\begin{remark}\rm
	We also have to take into account quadrature errors
	when evaluating (\ref{eq:stability,max}).
	This can be done in many ways; see, e.g., \cite{CDE}.
\end{remark}

\appendix

\section{Derivation of the methods}

This section contains some details
left out of the discussion of section \ref{sec:method}.

\subsection{The \protect\boldmath\mcgq\ method}

To rewrite the local problem in a more explicit form, let
$\{s_n\}_{n=0}^{q}$ be a set of nodal points on $[0,1]$, with $s_0=0$
and $s_q=1$. A good choice for the \cgq\ method is the
Lobatto points of $[0,1]$.
Now, let $\tau_{ij}$ be the linear mapping from the
interval $I_{ij}$ to $(0,1]$, defined by
\begin{equation}
  \tau_{ij}(t) = \frac{t-t_{i,j-1}}{t_{ij}-t_{i,j-1}},
  \label{eq:tau}
\end{equation}
and let $\{\lambda_n^{[q]}\}_{n=0}^{q}$ be the $\{s_n\}_{n=0}^{q}$
Lagrange basis functions for $\mathcal{P}^q([0,1])$ on $[0,1]$, i.e.,
\begin{equation}
  \lambda_{n}^{[q]}(s) = \frac
  { (s-s_0)   \cdots (s-s_{n-1})   (s-s_{n+1})   \cdots(s-s_{q}) }
  { (s_n-s_0) \cdots (s_n-s_{n-1}) (s_n-s_{n+1}) \cdots(s_n-s_q) }.
  \label{eq:lagrange}
\end{equation}
We can then express $U_i$ on $I_{ij}$ in the form
\begin{equation}
  U_i(t) = \sum_{n=0}^q \xi_{ijn} \lambda^{[q_{ij}]}_n(\tau_{ij}(t)),
  \label{eq:ansatz,cg}
\end{equation}
and choosing the $\lambda_m^{[q-1]}$ as test functions we can formulate
the local problem
(\ref{eq:fem,mcg,local}) as follows:
Find $\{\xi_{ijn}\}_{n=0}^{q_{ij}}$, with $\xi_{ij0}=\xi_{i,j-1,q_{i,j-1}}$, such that
for $m=0,\ldots,q_{ij}-1$
\begin{equation}
  \int_{I_{ij}}
  \sum_{n=0}^{q_{ij}} \xi_{ijn} \frac{d}{dt}
  \left[ \lambda^{[q_{ij}]}_n(\tau_{ij}(t)) \right]
  \lambda^{[q_{ij}-1]}_m(\tau_{ij}(t)) \ dt =
  \int_{I_{ij}} f_i(U(t),t) \lambda^{[q_{ij}-1]}_m(\tau_{ij}(t)) \ dt.
  \label{eq:fem,xi,cg}
\end{equation}
To simplify the notation, we drop the $_{ij}$ subindices and
assume that the time-interval is $[0,k]$, keeping in mind that, although not visible,
all other components are present in $f$. We thus seek to determine the
coefficients $\{\xi_n\}_{n=1}^{q}$ with $\xi_0$ given, such that for $m=1,\ldots,q$ we have
\begin{equation}
  \sum_{n=0}^q \xi_n
  \frac{1}{k}
  \int_0^k
  \dot{\lambda}^{[q]}_n(\tau(t))
  \lambda^{[q-1]}_{m-1}(\tau(t)) \ dt =
  \int_0^k
  f \lambda^{[q-1]}_{m-1}(\tau(t)) \ dt,
  \label{eq:fem,xi,cg,simple}
\end{equation}
or simply
\begin{equation}
  \sum_{n=1}^q a^{[q]}_{mn} \xi_n = b_m,
  \label{eq:fem,xi,cg,simpler}
\end{equation}
where
\begin{equation}
  a^{[q]}_{mn} =
  \int_0^1
  \dot{\lambda}^{[q]}_n(t)
  \lambda^{[q-1]}_{m-1}(t) \ dt
  \label{eq:a,mcg}
\end{equation}
and
\begin{equation}
	b_m = \int_0^k
  f \lambda^{[q-1]}_{m-1}(\tau(t)) \ dt -
  a_{m0} \xi_0.
\end{equation}

We explicitly compute the inverse
$\bar{A}^{[q]}=(\bar{a}^{[q]}_{mn})$
of the matrix $A^{[q]}=(a^{[q]}_{mn})$.
Thus, switching back to the full notation, we get
\begin{equation}
  \xi_{ijm} =
  -\xi_0 \sum_{n=1}^q \bar{a}_{mn}^{[q]} a_{n0} +
  \int_{I_{ij}} w_m^{[q_{ij}]}(\tau_{ij}(t)) \ f_i(U(t),t) \ dt, \qquad m = 1,\ldots,q_{ij},
  \label{eq:mcg,xi,expl,first}
\end{equation}
where the \emph{weight functions} $\{w_{m}^{[q]}\}_{m=1}^q$ are
given by
\begin{equation}
  w_m^{[q]} =
  \sum_{n=1}^q \bar{a}_{mn}^{[q]}
  \lambda_{n-1}^{[q-1]}, \qquad m = 1,\ldots,q.
  \label{eq:mcg,weights}
\end{equation}
Following Lemma \ref{lem:simpl,cg} below, this relation may be somewhat simplified.

\begin{lemma}
	\label{lem:simpl,cg}
	For the \mcgq\ method, we have
  \begin{displaymath}
    \sum_{n=1}^q \bar{a}_{mn}^{[q]} a_{n0} = -1.
  \end{displaymath}
\end{lemma}

\begin{proof}
  Assume the interval to be $[0,1]$.
  The value is independent of $f$ so we may take $f=0$. We thus want
  to prove that if $f=0$, then $\xi_n=\xi_0$ for $n=1,\ldots,q$, i.e.,
  $U=U_0$ on $[0,1]$ since $\{\lambda_n^{[q]}\}_{n=0}^{q}$ is a nodal
  basis for $\mathcal{P}^q([0,1])$.
  Going back to the Galerkin orthogonality
  (\ref{eq:fem,mcg,local,orthogonality}), this amounts to showing that if
  \begin{displaymath}
    \int_0^1 \dot{U} v \ dt = 0 \qquad \forall v\in\mathcal{P}^{q-1}([0,1]),
  \end{displaymath}
  with $U\in\mathcal{P}^q([0,1])$, then $U$ is constant on $[0,1]$.
  This follows by taking $v=\dot{U}$. Thus,
  $\xi_n=\xi_0$ for $n=1,\ldots,q$, so that
  the value of $\sum_{n=1}^q \bar{a}_{mn}^{[q]} a_{n0}$ must be $-1$.
  This completes the proof.
\qquad\end{proof}

The \mcgq\ method thus reads as follows:
For every local interval $I_{ij}$, find $\{\xi_{ijn}\}_{n=0}^{q_{ij}}$, with $\xi_{ij0}=\xi_{i,j-1,q_{i,j-1}}$, such that
\begin{equation}
  \xi_{ijm} =
  \xi_{ij0} +
  \int_{I_{ij}} w_m^{[q_{ij}]}(\tau_{ij}(t)) \ f_i(U(t),t) \ dt,
	\qquad m=1,\ldots,q_{ij},
  \label{eq:mcg,xi,expl,app}
\end{equation}
for certain weight functions
$\{w_{n}^{[q]}\}_{m=1}^q\subset\mathcal{P}^{q-1}(0,1)$, and where the initial condition is
specified by
$\xi_{i00} = u_i(0)$ for $i=1,\ldots,N$.

The weight functions may be computed analytically for small $q$, and
for general $q$ they are easy to compute numerically.

\subsection{The \protect\boldmath\mdgq\ method}

We now make the same ansatz as for the continuous method,
\begin{equation}
  U_i(t) = \sum_{n=0}^q \xi_{ijn} \lambda^{[q_{ij}]}_n(\tau_{ij}(t)),
  \label{eq:ansatz,dg}
\end{equation}
where the difference is that we now have $q+1$ degrees of freedom on
every interval, since we no longer have the continuity requirement for
the trial functions. We make the assumption that the nodal points for
the nodal basis functions are chosen so that
\begin{equation}
  s_q = 1,
  \label{eq:nodes,condition,dg}
\end{equation}
i.e., the end-point of every subinterval is a nodal point for the
basis functions.

With this ansatz, we get the following set of
equations for determining $\{\xi_{ijn}\}_{n=0}^{q_{ij}}$:
\begin{eqnarray}
\\
  \left(\sum_{n=0}^{q_{ij}}
  \xi_{ijn} \lambda^{[q_{ij}]}_n(0) - \xi_{ij0}^- \right) \lambda^{[q_{ij}]}_m(0) +
  \int_{I_{ij}}
  \sum_{n=0}^{q_{ij}} \xi_{ijn} \frac{d}{dt}
  \left[ \lambda^{[q_{ij}]}_n(\tau_{ij}(t)) \right]
  \lambda^{[q_{ij}]}_m(\tau_{ij}(t)) \ dt \nonumber\\ =
  \int_{I_{ij}} f_i(U(t),t) \lambda^{[q_{ij}]}_m(\tau_{ij}(t)) \ dt
  \label{eq:fem,xi,dg}
\end{eqnarray}
for $m=0,\ldots,q_{ij}$,
where we use $\xi_{ij0}^-$ to denote $\xi_{i,j-1,q_{i,j-1}}$, i.e., the value
at the right end-point of the previous interval.
To simplify the notation, we drop the subindices $_{ij}$ again and
rewrite to $[0,k]$.
We thus seek to determine the
coefficients $\{\xi_n\}_{n=0}^{q}$ such that for $m=0,\ldots,q$ we have
\begin{eqnarray}
\\
  \left( \sum_{n=0}^{q}
    \xi_n \lambda^{[q]}_n(0) - \xi_0^-
  \right) \lambda^{[q]}_m(0) +
  \sum_{n=0}^q \xi_n
  \frac{1}{k}
  \int_0^k
  \dot{\lambda}^{[q]}_n(\tau(t))
  \lambda^{[q]}_{m}(\tau(t)) \ dt =
  \int_0^k
  f \lambda^{[q]}_{m}(\tau(t)) \ dt,\nonumber
  \label{eq:fem,xi,dg,simple}
\end{eqnarray}
or simply
\begin{equation}
  \sum_{n=0}^q a^{[q]}_{mn} \xi_n = b_m,
  \label{eq:fem,xi,dg,simpler}
\end{equation}
where
\begin{equation}
  a^{[q]}_{mn} =
  \int_0^1
  \dot{\lambda}^{[q]}_n(t)
  \lambda^{[q]}_{m}(t) \ dt +
  \lambda^{[q]}_n(0) \lambda^{[q]}_m(0)
  \label{eq:a,mdg}
\end{equation}
and
\begin{equation}
  b^{[q]}_m = \int_0^k f \lambda^{[q]}_m(\tau(t)) \ dt +
  \xi_0^- \lambda_m^{[q]}(0).
  \label{eq:b,mdg}
\end{equation}

Now, let $A^{[q]}$ be the $(q+1)\times(q+1)$
matrix $A^{[q]}=( a^{[q]}_{mn})$ with inverse
$\bar{A}^{[q]}=( \bar{a}^{[q]}_{mn})$. Then, switching
back to the full notation, we have
\begin{equation}
  \xi_{ijm} =
  \xi_{ij0}^- \sum_{n=0}^q \bar{a}_{mn}^{[q]} \lambda_n^{[q]}(0) +
  \int_{I_{ij}} w_m^{[q_{ij}]}(\tau_{ij}(t)) \ f_i(U(t),t) \ dt, \qquad m = 0,\ldots,q_{ij},
  \label{eq:mdg,xi,expl,first}
\end{equation}
where the \emph{weight functions} $\{w_{n}^{[q]}\}_{n=0}^q$ are
given by
\begin{equation}
  w_m^{[q]} =
  \sum_{n=0}^q \bar{a}_{mn}^{[q]}
  \lambda_{n}^{[q]}, \qquad m = 0,\ldots,q.
  \label{eq:mdg,weights}
\end{equation}
As for the continuous method, this may be somewhat simplified.

\begin{lemma}
  For the \mdgq\ method, we have
  \begin{displaymath}
    \sum_{n=0}^q \bar{a}_{mn}^{[q]} \lambda_n^{[q]}(0) = 1.
  \end{displaymath}
\end{lemma}

\begin{proof}
  As in the proof for the \mcgq\ method, assume that the interval is
  $[0,1]$. Since the value of the expression is independent of $f$ we
  can take $f=0$. We thus want to prove that if $f=0$, then the
  solution $U$ is constant. By the Galerkin orthogonality, we have
  \begin{displaymath}
    [U]_0 v(0) + \int_0^1 \dot{U} v \ dt = 0 \qquad \forall v\in\mathcal{P}^q(0,1),
  \end{displaymath}
  with $U\in\mathcal{P}^q(0,1)$.
  Taking $v=U - U(0^-)$, we have
  \begin{displaymath}
    \begin{array}{rcl}
      0
      &=& ([U]_0)^2 + \int_0^1 \dot{U} (U-U(0^-)) \ dt
       = ([U]_0)^2 + \frac{1}{2}\int_0^1 \OD{}{t} (U-U(0^-))^2 \ dt \\
      &=& \frac{1}{2} (U(0^+)-U(0^-))^2 + \frac{1}{2} (U(1)-U(0^-))^2,
    \end{array}
  \end{displaymath}
  so that $[U]_0 = 0$. Now take $v=\dot{U}$. This
  gives $\int_0^1 (\dot{U})^2 \ dt = 0$. Since then both $[U]_0=0$ and
  $\dot{U}=0$ on $[0,1]$, $U$ is constant and equal to $U(0^-)$,
  and the proof is complete.
\qquad\end{proof}

The \mdgq\ method thus reads as follows:
For every local interval $I_{ij}$,
find $\{\xi_{ijn}\}_{n=0}^{q_{ij}}$, such that for $m=0,\ldots,q_{ij}$ we have
\begin{equation}
  \xi_{ijm} =
  \xi_{ij0}^- +
  \int_{I_{ij}} w_m^{[q_{ij}]}(\tau_{ij}(t)) \ f_i(U(t),t) \ dt
  \label{eq:mdg,xi,expl,app}
\end{equation}
for certain weight functions
$\{w_{n}^{[q]}\}_{n=0}^q\subset\mathcal{P}^{q}(0,1)$.

\end{document}